\documentclass[11pt, reqno]{amsart}
\pdfoutput=1

\usepackage[top=4cm,bottom=4cm,left=3.5cm,right=3.5cm]{geometry}

\usepackage{amssymb,amsmath,amsthm}
\usepackage[UKenglish]{babel}
\usepackage{enumitem}
\setlist{itemsep=0pt, topsep=0.5em}
\setlist[description]{leftmargin=*, itemsep = 5pt}
\usepackage{tikz-cd}
\usepackage{mathtools}
\usepackage{graphicx}
\usepackage{mathrsfs}
\usepackage{stmaryrd}
\usepackage[scr=euler]{mathalfa}
\usepackage[hyphens]{url}
\usepackage[breaklinks=true, hidelinks]{hyperref}

\providecommand{\noopsort}[1]{}

\theoremstyle{plain}
\newtheorem{theorem}{Theorem}[section]
\newtheorem{lemma}[theorem]{Lemma}
\newtheorem{proposition}[theorem]{Proposition}
\newtheorem{corollary}[theorem]{Corollary}

\newtheorem*{claim*}{Claim}

\newtheorem*{fact*}{Fact} 

\theoremstyle{definition}
\newtheorem{definition}[theorem]{Definition}
\newtheorem{remark}[theorem]{Remark}

\newtheorem{example}[theorem]{Example}

\newcommand{\Fraisse}{Fra\"{i}ss\'{e}}
\renewcommand{\epsilon}{\varepsilon}
\renewcommand{\theta}{\vartheta}
\renewcommand{\phi}{\varphi}
\newcommand{\cf}{cf.~}
\newcommand{\Cf}{Cf.~}
\newcommand{\eg}{e.g.~}
\newcommand{\ie}{i.e.~}
\newcommand{\iec}{i.e.,~}
\newcommand{\Iec}{I.e.,~}

\newcommand{\down}{{\downarrow}} 
\newcommand{\cvr}{\prec} 
\DeclareMathOperator{\htf}{\mathrm{ht}} 
\newcommand{\pit}{\pitchfork}
\newcommand{\br}[1]{\llbracket #1 \rrbracket} 
\renewcommand{\tilde}{\widetilde}
\renewcommand{\o}{\overline}
\DeclareMathOperator{\dom}{dom} 
\DeclareMathOperator{\cod}{cod}
\newcommand{\B}{\mathcal{B}} 
\newcommand{\U}{\mathcal{U}} 
\newcommand{\Sg}[1]{\mathscr{S}_{#1}} 
\newcommand{\co}[1]{#1_{\star}} 
\newcommand{\inc}[1]{\mathfrak{i}_{#1}} 
\newcommand{\sg}{\sigma}
\newcommand{\As}{A} 
\newcommand{\Bs}{B} 
\newcommand{\Cs}{C} 
\newcommand{\Ds}{D} 
\newcommand{\g}{\mathfrak{g}}
\DeclareMathOperator{\core}{core}

\newcommand{\Ek}{\mathbb{E}_k} 
\newcommand{\Mk}{\mathbb{M}_k} 
\DeclareMathOperator{\Emb}{\mathbb{S}} 
\DeclareMathOperator{\Path}{\mathbb{P}} 
\newcommand{\colim}{\operatornamewithlimits{colim}}
\newcommand{\LE}{L^{E}} 
\newcommand{\RE}{R^{E}} 

\DeclareMathOperator{\Set}{\mathbf{Set}} 
\DeclareMathOperator{\C}{\mathscr{C}} 
\DeclareMathOperator{\A}{\mathscr{A}} 
\DeclareMathOperator{\F}{\mathscr{F}} 
\DeclareMathOperator{\T}{\mathscr{T}} 
\DeclareMathOperator{\E}{\mathscr{E}} 
\newcommand{\CSstar}{\mathbf{Struct}_{\bullet}(\sg)} 
\newcommand{\CSplus}{\mathbf{Struct}(\sg^I)} 
\DeclareMathOperator{\D}{\mathscr{D}} 
\DeclareMathOperator{\Mod}{\mathbf{Mod}} 
\newcommand{\CS}{\mathbf{Struct}(\sg)}
\newcommand{\R}{\mathscr{R}} 
\newcommand{\RT}{\R^{E}} 
\newcommand{\RTk}{\R^{E}_{k}} 
\newcommand{\RM}{\R^{M}} 
\newcommand{\RMk}{\R^{M}_{k}} 
\newcommand{\Cp}{\C_{p}} 

\newcommand{\LL}{\mathcal{L}} 
\newcommand{\FO}{\mathrm{FO}} 
\newcommand{\EFO}{\exists^+\mathrm{FO}} 
\newcommand{\ML}{\mathrm{ML}} 
\newcommand{\prCk}{\rightarrow^{\C}_k} 
\newcommand{\eqaCk}{\rightleftarrows_{k}^{\C}} 
\newcommand{\eqbCk}{\leftrightarrow_{k}^{\C}} 
\newcommand{\eqcCk}{\cong_{k}^{\C}} 

\newcommand{\emb}{\rightarrowtail} 
\newcommand{\epi}{\twoheadrightarrow} 
\newcommand{\rl}{\rightleftarrows}
\newcommand{\id}{\mathrm{id}} 
\newcommand{\Q}{\mathscr{Q}} 
\newcommand{\M}{\mathscr{M}} 

\tikzcdset{
    relay arrow/.default = 10pt,
    relay arrow/.style = {
        rounded corners,
        to path = {
            \pgfextra{
                \def\sourcecoordinate{\pgfpointanchor{\tikztostart}{center}}
                \def\targetcoordinate{\pgfpointanchor{\tikztotarget}{center}}
                \pgfmathanglebetweenpoints{\sourcecoordinate}{\targetcoordinate}
                \edef\tempangle{\pgfmathresult}
                \pgftransformrotate{\tempangle}
                \pgfmathifthenelse{#1>0}{\tempangle+90}{\tempangle-90}
                \pgfcoordinate{tempcoord}{\pgfpointanchor{\tikztostart}{\pgfmathresult}}
            }
            (tempcoord)
            -- ([yshift=#1]tempcoord)
            -- ([yshift=#1]tempcoord-|\tikztotarget.center)\tikztonodes
            --(\tikztotarget)
        }
    }
}

\begin{document}
\title[\tiny{Arboreal Categories and Equi-resource Homomorphism Preservation Theorems}]{Arboreal Categories and Equi-resource Homomorphism Preservation Theorems}

\author{Samson Abramsky}
\address{Department of Computer Science, University College London, 66--72 Gower Street, London, WC1E 6EA, United Kingdom}\email{s.abramsky@ucl.ac.uk}

\author{Luca Reggio}
\address{Department of Computer Science, University College London, 66--72 Gower Street, London, WC1E 6EA, United Kingdom}
\email{l.reggio@ucl.ac.uk}

\thanks{Research supported by the EPSRC grant EP/V040944/1, and by the European Union's Horizon 2020 research and innovation programme under the Marie Sk{\l}odowska-Curie grant agreement No 837724.}

\begin{abstract}
The classical homomorphism preservation theorem, due to {\L}o{\'s}, Lyndon and Tarski, states that a first-order sentence $\phi$ is preserved under  homomorphisms between structures if, and only if, it is equivalent to an existential positive sentence~$\psi$.
Given a notion of (syntactic) complexity of sentences, an ``equi-resource'' homomorphism preservation theorem improves on the classical result by ensuring that $\psi$ can be chosen so that its complexity does not exceed that of $\phi$.

We describe an axiomatic approach to equi-resource homomorphism preservation theorems based on the notion of arboreal category. This framework is then employed to establish novel homomorphism preservation results, and improve on known ones, for various logic fragments, including first-order, guarded and modal logics.
\end{abstract}

\maketitle

\section{Introduction}
%

\subsection{Summary}
In this paper, we take steps towards an \emph{axiomatic model theory} which is resource-sensitive, and well adapted for finite and algorithmic model theory and descriptive complexity. 
This builds on previous work on game comonads \cite{abramsky2017pebbling,DBLP:conf/csl/AbramskyS18,conghaile2021game,Guarded2021,DJR2021,Hybrid2022,FVM2022}, which describe model comparison games in terms of resource-indexed comonads on the category of relational structures. This allows the syntax-free description of a wide range of logics, indexed by resource parameters such as quantifier rank and number of variables. As explained in \cite{emerging2022}, this opens up a landscape in which properties of the comonad are related to the degree of tractability of the logic.

Game comonads and the associated (comonadic) adjunctions have been axiomatised at a general categorical level in terms of \emph{arboreal categories} \cite{AR2021icalp,AR2022}, in which the objects have an intrinsic tree structure. 
This structure can be used to define notions of bisimulation, game, etc.\ at an abstract level, and to transfer these process notions via an adjunction to an \emph{extensional category}---the main concrete example in applications being the category of relational structures.

In this paper, we show that this axiomatic framework can be used to prove a substantial resource-sensitive model-theoretic result, Rossman's equirank homomorphism preservation theorem, as well as a number of extensions and variations, \eg to modal and guarded logics, and relativisations to a wide range of subclasses of structures.

Some striking features of this axiomatic approach are its generality, its resource sensitivity, and the fact that it does not use compactness.
It is interesting to compare this with the use of accessible categories as an axiomatic framework for the study of abstract elementary classes \cite{beke2012abstract}. This is aimed at extending first-order model theory into the infinite, replacing compactness by $\lambda$-accessibility. By contrast, motivated by the needs of finite model theory and descriptive complexity, we are interested in capturing fine structure ``down below'', in fragments of first-order logic.\footnote{Monadic second-order logic is also of interest and within scope, see \cite{FVM2022}.}

\subsection{Extended discussion}
Homomorphism preservation theorems relate the syntactic shape of a sentence with the semantic property of being preserved under homomorphisms between structures. 
Recall that a first-order sentence $\phi$ in a vocabulary $\tau$ is said to be \emph{preserved under homomorphisms} if, whenever there is a homomorphism of $\tau$-structures $\As\to \Bs$ and ${\As\vDash \phi}$, then also $\Bs\vDash\phi$.
Further, an \emph{existential positive sentence} is a first-order sentence that uses only the connectives $\vee,\wedge$ and the quantifier $\exists$.
The following classical result, known as the \emph{homomorphism preservation theorem}, is due to {\L}o{\'s}, Lyndon and Tarski \cite{Los1955, Lyndon1959, Tarski1955} and applies to arbitrary (first-order) vocabularies.
\begin{theorem}\label{th:HPT}
A first-order sentence is preserved under homomorphisms if, and only if, it is equivalent to an existential positive sentence.
\end{theorem} 

The homomorphism preservation theorem is a fairly straightforward consequence of the compactness theorem, see \eg \cite[Lemma~3.1.2]{TZ2012}. However, applying the compactness theorem means that we lose control over the syntactic shape of an existential positive sentence $\psi$ that is equivalent to a sentence $\phi$ preserved under homomorphisms. In particular, it is an ineffective approach if we want to determine to which extent the passage from $\phi$ to $\psi$ increases the ``complexity''.

One way to measure the complexity of a formula $\phi$ is in terms of its \emph{quantifier rank}, \ie the maximum number of nested quantifiers appearing in $\phi$. 
Rossman's \emph{equirank homomorphism preservation theorem}~\cite{Rossman2008}, which applies to relational vocabularies (\ie vocabularies that contain no constant or function symbols), shows that it is possible to find a $\psi$ whose quantifier rank is less than or equal to the quantifier rank of~$\phi$:
\begin{theorem}\label{th:equirank-HPT}
A first-order sentence of quantifier rank at most $k$ is preserved under homomorphisms if, and only if, it is equivalent to an existential positive sentence of quantifier rank at most $k$.
\end{theorem} 
This is a considerable improvement on the classical homomorphism preservation theorem and was proved by Rossman on the way to his celebrated \emph{finite homomorphism preservation theorem}, stating that Theorem~\ref{th:HPT} admits a relativisation to finite structures.\footnote{The finite homomorphism preservation theorem is a major result in finite model theory, as well as a surprising one given that most preservation theorems fail when restricted to finite structures. Note that the finite homomorphism preservation theorem and the classical one are incomparable results.} In the proof of the equirank homomorphism preservation theorem, the application of the compactness theorem is replaced with a model construction which is similar in spirit to the construction of a saturated elementary extension of a given structure.

The main contribution of this paper consists in laying out a categorical framework in which ``equi-resource'' homomorphism preservation theorems can be proved in an axiomatic fashion. 
In \cite{abramsky2017pebbling,DBLP:conf/csl/AbramskyS18}, \emph{game comonads} were introduced to capture in a structural way a number of logic fragments and corresponding combinatorial and game-theoretic notions, both at the level of finite and infinite structures. For a recent survey article, see~\cite{emerging2022}. The template of game comonads was axiomatised in~\cite{AR2021icalp}, see also~\cite{AR2022}, by means of the notion of \emph{arboreal category}.
Our proof strategy consists in establishing abstract homomorphism preservation theorems at the general level of arboreal categories, and then instantiating these results for specific choices of game comonads. 

We thus obtain novel equi-resource homomorphism preservation theorems for modal and guarded logics (Theorems~\ref{th:hpt-graded-modal-logic} and~\ref{th:hpt-guarded}, respectively), along with relativisations to appropriate subclasses of structures---\eg the class of finite structures (Theorems~\ref{th:hpt-graded-modal-logic-finite} and~\ref{th:hpt-guarded-logics-finite}, respectively). Further, we derive a relativisation result (Theorem~\ref{t:equirank-hpt-relative}) which refines Rossman's equirank homomorphism preservation theorem.

This paper is organised as follows. In Section~\ref{s:prelim-game-comonads} we provide a brief introduction to game comonads, and in Section~\ref{s:prelim-arboreal} we recall the necessary definitions and facts concerning arboreal categories. Homomorphism preservation theorems are recast into categorical statements and proved at the level of arboreal categories in Sections~\ref{s:logics-HPTs},~\ref{s:exploring-the-landscape} and~\ref{s:axiomatic}. Section~\ref{s:proof-mc} contains the proof of our main technical result, namely Theorem~\ref{t:model-construction}. Finally, in Section~\ref{s:well-behaved} we present a simple class of relativisation results for classes of well-behaved structures; these are discussed separately as they are not \emph{equi-resource} homomorphism preservation theorems.

Throughout this article, we shall assume the reader is familiar with the basic notions of category theory; standard references include \eg \cite{adamek2004abstract,MacLane}.

\vspace{1em}
Let us conclude this introduction with some observations to guide the reader through the article:
\begin{description}
\item[From games to arboreal categories] This paper builds on our prior work~\cite{AR2022} on arboreal categories. Familiarity with the latter work would be beneficial, but we do recall the necessary material in Section~\ref{s:prelim-arboreal} in order to make the article self-contained. Let us briefly overview the different levels of abstractions leading from model-comparison games to arboreal categories.

Preservation of many fragments of first-order logic, and variations thereof, can be described in terms of the existence of a winning strategy in a game. Game comonads arise from the insight that the construction associating to a structure~$A$ its collection~$GA$ of ``explorations'' in the game induces a comonad~$G$, which is an enrichment of the \mbox{(finite-)list} comonad. The coalgebras for the comonad~$G$ capture compatible forest covers of structures, which can be thought of as unfoldings of games, and winning strategies in the game correspond to various types of morphisms between coalgebras.

Arboreal categories capture the main features of the categories of coalgebras for game comonads, and the minimal  properties required to fruitfully imitate the notions of back-and-forth equivalence and bisimulation, and to show that they coincide, gave rise to the axioms of arboreal categories. The filtering of coalgebra covers corresponding to the stratification induced by the resource parameter in the logic (\eg the quantifier rank of formulas) begets resource-indexed arboreal adjunctions.

\item[A dividing line] Analysing homomorphism preservation theorems from a categorical perspective, we identify the \emph{bisimilar companion property} as a key dividing line, and Section~\ref{s:exploring-the-landscape} is divided accordingly. Suppose that $G$ is a game comonad on a category of structures corresponding to a logic fragment $\LL$. Typically, any structure $A$ is equivalent to $GA$ in the \emph{existential positive} fragment of $\LL$, but there could be a sentence in $\LL$ that is valid in precisely one of the two structures. The bisimilar companion property asserts that this is not the case: $A$ and $GA$ are indistinguishable in the logic fragment~$\LL$.

The bisimilar companion property is satisfied \eg by the comonads corresponding to modal and guarded logics.  
For these logics, we obtain simple proofs of equi-resource homomorphism preservation theorems.

On the other hand, the bisimilar companion property fails for first-order logic with bounded quantifier rank, or with a bounded number of variables (the latter falls outside of the scope of this paper, see Remark~\ref{rem:does-not-apply-to}), and showing that this property can be ``forced'' requires a considerable amount of work. Our approach is inspired by that of Rossman~\cite{Rossman2008}, and shows that the methods and ideas put forward in \emph{op.\ cit.} can be lifted to the axiomatic setting. There is a difference though: whereas Rossman uses \emph{$n$-cores}, approximating the notion of \emph{core}~\cite{HN1992}, we use cofree coalgebras. The former are (minimal) retracts of structures, while the latter can be thought of as ``coverings''.

\item[Equality] The framework of game comonads typically encodes the preservation of \emph{equality-free} fragments of logics. Generally, the difference is invisible when dealing with existential positive fragments, or extensions with counting quantifiers, as these often admit equality elimination (\cf \cite[\S VI]{DJR2021} for the case of logics with counting quantifiers). The difference is also immaterial for modal logic, because its standard translation is contained in the equality-free fragment of first-order logic. Nevertheless, this gap needs to be addressed when considering \eg first-order logic with bounded quantifier rank.

This was handled in \cite{abramsky2017pebbling,AS2021} and subsequent works by considering a fresh binary relation $I$ and assuming that (i)~$I$ is interpreted as the identity relation on the class of structures we are interested in, and (ii)~all homomorphisms preserve~$I$. Note however that, if $A$ is a structure in which $I$ coincides with the identity relation, the structure $GA$ obtained by applying a game comonad $G$ to $A$ need not have the same property. 

In the present paper, we adopt the approach outlined in~\cite{AR2022}, which is better suited for an axiomatic treatment. This consists in composing the comonadic adjunction induced by a game comonad with an adjunction that ``introduces and eliminates $I$'' (see Example~\ref{ex:res-ind-arb-adj} for more details). This is why, throughout, we need to work with possibly non-comonadic adjunctions.
\end{description}

\section{Logic fragments and game comonads}\label{s:prelim-game-comonads}

We shall mainly deal with two types of vocabularies:
\begin{itemize}
\item \emph{Relational vocabularies}, \ie first-order vocabularies that contain no function or constant symbols.
\item \emph{Multi-modal vocabularies}, \ie relational vocabularies in which every relation symbol has arity at most $2$.
\end{itemize} 

Multi-modal vocabularies will be referred to simply as \emph{modal vocabularies}. If $\sg$ is a modal vocabulary, we can assign to each unary relation symbol $P\in \sg$ a propositional variable $p$, and to each binary relation symbol $R\in \sg$ modalities $\Diamond_R$ and $\Box_R$. We refer to $\sg$-structures as \emph{Kripke structures}. For any Kripke structure $\As$, the interpretation of $P$ in $\As$, denoted by $P^{\As}$, corresponds to the valuation of the propositional variable $p$, and the binary relation $R^{\As}$ to the accessibility relation for the modalities $\Diamond_R$ and $\Box_R$. 

For our running examples and intended applications, we will be interested in the following resource-bounded fragments of first-order logic (always including the equality symbol) and modal logic, for a relational vocabulary $\sg$ and positive integer~$k$:
\begin{itemize}
\item $\FO_k$ and $\EFO_k$: these denote, respectively, the set of sentences of quantifier rank $\leq k$ in the vocabulary $\sg$, and its existential positive fragment.
\item $\ML_k$ and $\exists^+\ML_k$: if $\sg$ is a modal vocabulary, $\ML_k$ is the set of modal formulas of modal depth $\leq k$ in the vocabulary $\sg$ (recall that the \emph{modal depth} is the maximum number of nested modalities in a formula). Moreover, $\exists^+\ML_k$ denotes the existential positive fragment of $\ML_k$, \ie the set of formulas in $\ML_k$ that use only the connectives $\vee,\wedge$ and the diamond modalities.
\item $\ML_k(\#)$: this is the extension of $\ML_k$ with \emph{graded modalities}. Recall that graded modalities have the form $\Diamond_R^n$ and $\Box_R^n$, with $n\geq 0$, and $\As, a \models {\Diamond_R^n} \, \phi$ if there are at least $n$ $R$-successors of $a$ satisfying $\phi$; $\Box_R^n\phi$ is logically equivalent to $\neg \Diamond_R^n \neg \phi$.
\end{itemize}
Further logics, namely \emph{guarded logics}, will be considered in Section~\ref{ss:tame}.

Any logic fragment\footnote{We employ the nomenclature ``logic fragment'', rather than the more customary term ``theory'', as we are mainly interested in the situation where $\LL$ is defined by constraining the syntactic shape~of~sentences.} $\LL$, \ie any set of first-order sentences in a vocabulary $\tau$, induces an equivalence relation $\equiv^\LL$ on $\tau$-structures defined as follows: for all $\tau$-structures $\As, \Bs$,
\[
\As \equiv^{\LL} \Bs \ \ \Longleftrightarrow \ \ \forall \phi\in\LL. \, (\As\vDash \phi \, \Leftrightarrow \, \Bs\vDash \phi).
\]
If $\LL$ consists of all first-order sentences, $\equiv^\LL$ coincides with elementary equivalence. 

Syntax-free characterisations of the equivalence relations $\equiv^\LL$ play an important role in model theory. For example, the Keisler-Shelah theorem states that two $\tau$-structures are elementarily equivalent if, and only if, they have isomorphic ultrapowers. A different approach is through \emph{model comparison games}. These have a wide range of applications in model theory, see \eg \cite[\S 3]{Hod93}, and are central to finite model theory where tools such as the compactness theorem and ultraproducts are not available. Model comparison games lead to a perspective which may be described as ``model theory without compactness''. 

Game comonads arise from the insight that model comparison games can be seen as semantic constructions in their own right. Although we shall not employ games as a tool, we recall two examples of games to motivate the framework of game comonads.

Henceforth we shall work with a relational vocabulary $\sg$.
Let $\As,\Bs$ be $\sigma$-structures. Both types of game are two-player games played between Spoiler and Duplicator. Whereas Spoiler aims to show that $\As$ and $\Bs$ are different, Duplicator aims to show that they are similar. Each game is played in a number of rounds:
\begin{description}
\item[Ehrenfeucht-\Fraisse~game] In the $i$th round, Spoiler chooses an element in one of the structures and Duplicator responds by choosing an element in the other structure. After $k$ rounds have been played, we have sequences $[a_1, \ldots, a_k]$ and $[b_1, \ldots, b_k]$ of elements from $\As$ and $\Bs$ respectively. Duplicator wins this play if the ensuing relation $r\coloneqq \{(a_i, b_i) \mid 1 \leq i \leq k \}$ is a partial isomorphism. Duplicator wins the $k$-round game if they have a strategy which is winning after $i$ rounds, for all $1\leq i\leq k$.
\item[Bisimulation game] Suppose $\sg$ is a modal vocabulary. The game is played between pointed Kripke structures $(\As, a)$ and $(\Bs, b)$, where $a\in \As$ and $b\in \Bs$. The initial position is $(a_0,b_0)=(a,b)$. In the $i$th round, with current position $(a_{i-1},b_{i-1})$, Spoiler chooses one of the two structures, say $\As$, a binary relation symbol $R$ in $\sg$, and an element $a_{i}\in \As$ such that $(a_{i-1},a_i)\in R^{\As}$. Duplicator responds by choosing an element in the other structure, say $b_i\in \Bs$, such that $(b_{i-1},b_i)\in R^{\Bs}$. If Duplicator has no such response, they lose. Duplicator wins the $k$-round game if, for all unary relation symbols $P$ in $\sg$, we have $P^{\As}(a_i) \Leftrightarrow P^{\Bs}(b_i)$ for all $0\leq i\leq k$.
\end{description}

Assume the vocabulary $\sg$ is finite. The classical Ehrenfeucht-\Fraisse~theorem \cite{Ehrenfeucht1960,Fraisse1954} states that Duplicator has a winning strategy in the $k$-round Ehrenfeucht-\Fraisse~game played between $\As$ and $\Bs$ if, and only if, $\As$ and $\Bs$ satisfy the same first-order sentences of quantifier rank at most $k$. Similarly, Duplicator has a winning strategy in the $k$-round bisimulation game played between pointed Kripke structures $(\As,a)$ and $(\Bs, b)$ if, and only if, $(\As,a)$ and $(\Bs, b)$ satisfy the same modal formulas of modal depth at most $k$~\cite{HM1980}.

\begin{remark}
In both Ehrenfeucht-\Fraisse~and bisimulation games, the resource parameter is the number of rounds. This need not be the case in general. For instance, the resource parameter for pebble games \cite{Barwise1977,Immerman1982}, which correspond to finite-variable fragments of first-order logic, is the number of pebbles available to the players.
\end{remark}

Next, we introduce the comonads corresponding to Ehrenfeucht-\Fraisse~and bisimulation games, respectively.
For each $\sg$-structure $\As$, denote by $\Ek(\As)$ the set of all non-empty sequences of length at most $k$ of elements from $\As$. In other words, $\Ek(\As)$ is the set of all plays in $\As$ in the $k$-round Ehrenfeucht-\Fraisse~game. The interpretations of the relation symbols can be lifted from $\As$ to $\Ek(\As)$ as follows. Let $\epsilon_{\As}\colon \Ek(\As) \to \As$ be the function sending a sequence to its last element. For each relation symbol $R\in\sg$ of arity~$j$, we define $R^{\Ek(\As)}$ to be the set of all tuples $(s_1,\ldots,s_j)\in \Ek(\As)^j$ such that:
\begin{enumerate}[label=(\roman*)]
\item The sequences $s_1,\ldots,s_j$ are pairwise comparable in the prefix order.
\item $(\epsilon_{\As}(s_1),\ldots,\epsilon_{\As}(s_j))\in R^{\As}$.
\end{enumerate}
For every homomorphism $f\colon \As\to \Bs$, let 
\[
\Ek(f)\colon \Ek(\As)\to \Ek(\Bs), \ \ [a_1,\ldots,a_l]\mapsto [f(a_1),\ldots, f(a_l)]. 
\]
This yields a comonad (in fact, a \emph{family} of comonads indexed by $k>0$) on the category $\CS$ of $\sg$-structures and their homomorphisms, known as \emph{Ehrenfeucht-\Fraisse~comonad}~\cite{DBLP:conf/csl/AbramskyS18}. The underlying functor of this comonad is $\Ek\colon \CS\to \CS$, the counit is $\epsilon$, and the comultiplication at $\As$ is the homomorphism 
\[
\Ek(\As)\to \Ek\Ek(\As), \ \ [a_1,\ldots,a_l]\mapsto [[a_1],[a_1,a_2],\ldots,[a_1,\ldots,a_l]].
\]

A similar construction applies to $k$-round bisimulation games. 
Suppose $\sg$ is a modal vocabulary and let $(\As,a)$ be a pointed Kripke structure. We define a Kripke structure $\Mk(\As,a)$ whose carrier is the set of all paths $p$ of length at most $k$ starting from~$a$:
\[ a \xrightarrow{R_1} a_1 \xrightarrow{R_2} a_2 \to \cdots \xrightarrow{R_n} a_n \]
where $R_1, \dots, R_n$ are binary relation symbols in $\sg$.  
If $P\in\sg$ is unary then $P^{\Mk(\As,a)}$ is the set of paths $p$ whose last element $a_n$ belongs to $P^{\As}$. For a binary relation symbol $R$ in $\sg$, $R^{\Mk(\As,a)}$ is the set of pairs of paths $(p,p')$ such that $p'$ is obtained by extending $p$ by one step along $R$. The distinguished element of the Kripke structure $\Mk(\As,a)$ is the trivial path $\langle a\rangle$ of length~$0$, and the function $\epsilon_{(\As,a)}\colon (\Mk(\As,a),\langle a\rangle) \to (\As,a)$ sending a path to its last element is a morphism of pointed Kripke structures. By similar reasoning to the one above, the assignment $(\As,a)\mapsto (\Mk(\As,a),\langle a\rangle)$ can be upgraded to a comonad on the category $\CSstar$ of pointed Kripke structures and their homomorphisms, the \emph{modal comonad}~\cite{DBLP:conf/csl/AbramskyS18}.

In addition to the examples mentioned above, the framework of game comonads covers a number of model comparison games, \cf \eg \cite{abramsky2017pebbling,Guarded2021,Hybrid2022,conghaile2021game,FVM2022,Paine2020}. In each case, they yield structural (syntax-free) characterisations of equivalence in the corresponding logic fragments. This will be illustrated from an axiomatic perspective in Section~\ref{s:prelim-resource-ind-arb-adj}.

\section{Arboreal categories}\label{s:prelim-arboreal}

In this section, we recall from~\cite{AR2022} the basic definitions and facts concerning arboreal categories.  
All categories under consideration are assumed to be locally small and \emph{well-powered}, \ie every object has a \emph{small} set of subobjects (as opposed to a proper class).

\subsection{Proper factorisation systems}
Given arrows $e$ and $m$ in a category $\C$, we say that $e$ has the \emph{left lifting property} with respect to $m$, or that $m$ has the \emph{right lifting property} with respect to $e$, if for every commutative square as on the left-hand side~below
\begin{equation*}
\begin{tikzcd}
{\cdot} \arrow{d}[swap]{e} \arrow{r} & {\cdot} \arrow{d}{m} \\
{\cdot} \arrow{r} & {\cdot}
\end{tikzcd}
\ \ \ \ \ \ \ \ \ \ \ \ \ 
\begin{tikzcd}
{\cdot} \arrow{d}[swap]{e} \arrow{r} & {\cdot} \arrow{d}{m} \arrow[leftarrow]{dl}[description]{d} \\
{\cdot} \arrow{r} & {\cdot}
\end{tikzcd}
\end{equation*}
there is an arrow $d$ such that the right-hand diagram above commutes. 
For any class $\mathscr{H}$ of morphisms in $\C$, let ${}^{\pit}\mathscr{H}$ (respectively $\mathscr{H}^{\pit}$) be the class of morphisms having the left (respectively right) lifting property with respect to every morphism in $\mathscr{H}$.

\begin{definition}\label{def:weak-f-s}
A pair of classes of morphisms $(\Q,\M)$ in a category $\C$ is a \emph{weak factorisation system} provided it satisfies the following conditions:
\begin{enumerate}[label=(\roman*)]
\item Every morphism $f$ in $\C$ can be written as $f = m \circ e$ with $e\in \Q$ and $m\in \M$.
\item $\Q={}^{\pit}\M$ and $\M=\Q^{\pit}$.
\end{enumerate}
A \emph{proper factorisation system} is a weak factorisation system $(\Q,\M)$ such that each arrow in $\Q$ is epic and each arrow in $\M$ is monic. 
We refer to $\M$-morphisms as \emph{embeddings} and denote them by $\emb$. $\Q$-morphisms will be referred to as \emph{quotients} and denoted by~$\epi$.
\end{definition} 

Next, we state some well known properties of proper factorisation systems (\cf \eg \cite{freyd1972categories} or~\cite{riehl2008factorization}) which will be used throughout the paper without further reference.
\begin{lemma}\label{l:factorisation-properties}
Let $(\Q,\M)$ be a proper factorisation system in $\C$. The following hold:
\begin{enumerate}[label=(\alph*)]
\item\label{compositions} $\Q$ and $\M$ are closed under compositions.
\item\label{isos} $\Q$ contains all retractions, $\M$ contains all sections, and $\Q\cap\M=\{\text{isomorphisms}\}$.
\item\label{pullbacks} The pullback of an $\M$-morphism along any morphism, if it exists, is in $\M$.
\item\label{cancellation-e} $g\circ f\in \Q$ implies $g\in\Q$.
\item\label{cancellation-m} $g\circ f\in\M$ implies $f\in\M$.
\end{enumerate}
\end{lemma}

Assume $\C$ is a category admitting a proper factorisation system $(\Q,\M)$. In the same way that one usually defines the poset of subobjects of a given object $X\in\C$, we can define the poset $\Emb{X}$ of $\M$-subobjects of $X$. Given embeddings $m\colon S\emb X$ and $n\colon T\emb X$, let us say that $m\trianglelefteq n$ provided there is a morphism $i\colon S\to T$ such that $m=n\circ i$ (if it exists, $i$ is necessarily an embedding).
\[\begin{tikzcd}
S \arrow[rightarrowtail]{r}{m} \arrow[dashed, swap]{d}{i} & X \\
T \arrow[rightarrowtail]{ur}[swap]{n} & {}
\end{tikzcd}\] 
This yields a preorder on the class of all embeddings with codomain $X$. The symmetrisation~$\sim$ of~$\trianglelefteq$ is given by $m\sim n$ if, and only if, there exists an isomorphism $i\colon S\to T$ such that $m=n\circ i$. Let $\Emb{X}$ be the class of $\sim$-equivalence classes of embeddings with codomain $X$, equipped with the natural partial order $\leq$ induced by~$\trianglelefteq$. We shall systematically represent a $\sim$-equivalence class by any of its representatives. As $\C$ is well-powered and every embedding is a monomorphism, $\Emb{X}$ is a set.

For any morphism $f\colon X\to Y$ and embedding $m\colon S\emb X$, consider the (quotient, embedding) factorisation of $f\circ m$: 
\[\begin{tikzcd}
S \arrow[twoheadrightarrow]{r} & \exists_f S \arrow[rightarrowtail]{r}{\exists_f m} & Y.
\end{tikzcd}\]
This yields a monotone map $\exists_f\colon \Emb{X}\to\Emb{Y}$ sending $m$ to $\exists_f m$. Note that the map $\exists_f$ is well-defined because factorisations are unique up to isomorphism. If $f$ is an embedding (or, more generally, $f\circ m$ is an embedding), $\exists_f m$ can be identified with $f\circ m$. For the following observation, \cf \eg \cite[Lemma~2.7(a)]{AR2022}.

\begin{lemma}\label{l:emb-quo-order-embeddings}
Let $\C$ be a category equipped with a proper factorisation system and let $f\colon X\emb Y$ be an embedding in $\C$. Then $\exists_f \colon \Emb{X}\to \Emb{Y}$ is an order-embedding.
\end{lemma}

\subsection{Paths and arboreal categories}
Let $\C$ be a category endowed with a proper factorisation system $(\Q,\M)$.
\begin{definition}
An object $X$ of $\C$ is called a \emph{path} provided the poset $\Emb{X}$ is a finite linear order. Paths will be denoted by $P,Q$ and variations thereof.
\end{definition}

The collection of paths is closed under embeddings and quotients. That is, given an arrow $f\colon X\to Y$, if $f$ is an embedding and $Y$ is a path then $X$ is a path, and if $f$ is a quotient and $X$ is a path then $Y$ is a path~\cite[Lemma~3.5]{AR2022}.  

A \emph{path embedding} is an embedding $P\emb X$ whose domain is a path. 
We let $\Path{X}$ denote the sub-poset of $\Emb{X}$ consisting of the path embeddings. 
Because paths are closed under quotients, for any arrow $f\colon X\to Y$ the monotone map $\exists_f \colon \Emb{X}\to \Emb{Y}$ restricts to a monotone map
\begin{equation}\label{eq:Path-functor}
\Path{f}\colon \Path{X}\to\Path{Y}, \ \ (m\colon P\emb X)\mapsto (\exists_f m\colon \exists_f P\emb Y).
\end{equation}

Henceforth, we shall assume that $\C$ contains, up to isomorphism, only a set of paths. Thus, whenever convenient, we shall work with a small skeleton of the full subcategory of paths.
For any object $X$ of $\C$, we have a diagram with vertex $X$ consisting of all path embeddings with codomain $X$; the morphisms between paths are those which make the obvious triangles commute (note that they are necessarily embeddings):
\[\begin{tikzcd}[column sep=1.2em, row sep=2em]
 & X & \\
 P \arrow[bend left=20,rightarrowtail]{ur} \arrow[rightarrowtail]{rr} & & Q \arrow[bend right=20,rightarrowtail]{ul}
\end{tikzcd}\]
This yields a cocone over the small diagram $\Path{X}$. We say that $X$ is \emph{path-generated} provided this is a colimit cocone in $\C$.

Suppose for a moment that coproducts of sets of paths exist in $\C$. An object $X$ of $\C$ is \emph{connected} if, for all non-empty sets of paths $\{P_i\mid i\in I\}$ in~$\C$, any morphism 
\[
X\to \coprod_{i\in I}{P_i}
\] 
factors through some coproduct injection $P_j\to \coprod_{i\in I}{P_i}$.

In order to state the definition of arboreal category, let us say that a proper factorisation system is \emph{stable} if, for any quotient $e$ and embedding $m$ with common codomain, the pullback of $e$ along $m$ exists and is a quotient.

\begin{definition}
An \emph{arboreal category} is a category $\C$, equipped with a stable proper factorisation system, that satisfies the following conditions:
\begin{enumerate}[label=(\roman*)]
\item\label{ax:colimits} $\C$ has all coproducts of sets of paths.
\item\label{ax:2-out-of-3} For any paths $P,Q,Q'$ in $\C$, if a composite $P\to Q \to Q'$ is a quotient then so is $P\to Q$.
\item\label{ax:path-generated} Every object of $\C$ is path-generated. 
\item\label{ax:connected} Every path in $\C$ is connected.
\end{enumerate}
\end{definition}

\begin{remark}
Item~\ref{ax:2-out-of-3} in the previous definition is equivalent to the following \emph{2-out-of-3 condition}: For any paths $P,Q,Q'$ and morphisms
\[\begin{tikzcd}
P \arrow{r}{f} & Q \arrow{r}{g} & Q',
\end{tikzcd}\]
if any two of $f$, $g$, and $g\circ f$ are quotients, then so is the third. See \cite[Remark~3.9]{AR2022}. Moreover, item~\ref{ax:path-generated} is equivalent to saying that the inclusion functor $\Cp\hookrightarrow \C$ is dense, where $\Cp$ is the full subcategory of $\C$ defined by the paths \cite[Lemma~5.1]{AR2022}.

Finally, note that any arboreal category admits an initial object, obtained as the coproduct of the empty set, and any initial object is a path because its poset of $\M$-subobjects has a single element---namely, the equivalence class of the identity.
\end{remark}

If $(P, {\leq})$ is a poset, then $C \subseteq P$ is a \emph{chain}  if it is linearly ordered. $(P,\leq)$ is a \emph{forest} if, for all $x\in P$, the set $\down x\coloneqq \{y\in P\mid y\leq x\}$ is a finite chain. 
The \emph{height} of a forest is the supremum of the cardinalities of its chains.
The \emph{covering relation} $\cvr$ associated with a partial order $\leq$ is defined by $u\cvr v$ if and only if $u<v$ and there is no $w$ such that $u<w< v$. 
The \emph{roots} of a forest are the minimal elements, and a \emph{tree} is a forest with at most one root. 
Morphisms of forests are maps that preserve roots and the covering relation. 
The category of forests is denoted by $\F$, and the full subcategory of trees by~$\T$. 

\begin{example}\label{ex:RE}
Let $\sg$ be a relational vocabulary.
A \emph{forest-ordered $\sg$-structure} $(\As, {\leq})$ is a $\sg$-structure $\As$ equipped with a forest order $\leq$.
A morphism of forest-ordered $\sg$-structures $f\colon (\As, {\leq}) \to (\Bs, {\leq'})$ is a $\sg$-homomorphism $f\colon \As \to \Bs$ that is also a forest morphism. This determines a category $\R(\sg)$. 
We equip $\R(\sg)$ with the factorisation system given by (surjective morphisms, embeddings), where an embedding is a morphism which is an embedding \emph{qua} $\sg$-homomorphism.
Let $\RT(\sg)$ be the full subcategory of $\R(\sg)$ determined by those objects $(\As, {\leq})$ satisfying the following condition: 
\begin{enumerate}[label=\textnormal{(E)}]
\item\label{E} If $a,b\in\As$ are distinct elements that appear in a tuple of related elements $(a_1,\ldots,a_l)\in R^{\As}$ for some $R\in\sg$, then either $a<b$ or $b<a$.\footnote{\Iec if $a$ and $b$ are adjacent in the \emph{Gaifman graph} of $\As$, then they are comparable in the forest order.}
\end{enumerate}
For each $k>0$, let $\RTk(\sg)$ be the full subcategory of $\RT(\sg)$ of forest-ordered structures of height $\leq k$. In \cite[Theorem~9.1]{AS2021}, it is shown that $\RTk(\sg)$ is isomorphic to the category of coalgebras for the Ehrenfeucht-\Fraisse~comonad $\Ek$ on $\CS$. The objects $(\As, {\leq})$ of $\RTk(\sg)$ are forest covers of $\As$ witnessing that its \emph{tree-depth} is at most $k$ \cite{nevsetvril2006tree}.

The category $\RT(\sg)$ is arboreal when equipped with the restriction of the factorisation system on $\R(\sg)$. The paths in $\RT(\sg)$ are those objects in which the order is a finite chain. Similarly, $\RTk(\sg)$ is an arboreal category for all $k>0$, when equipped with the restriction of the factorisation system on $\R(\sg)$. See~\cite[Examples~5.4]{AR2022}.
\end{example}

\begin{example}\label{ex:RM}
Assume that $\sg$ is a modal vocabulary. Let $\RM(\sg)$ be the full subcategory of $\R(\sg)$ consisting of the tree-ordered $\sg$-structures $(\As, {\leq})$ satisfying: 
\begin{enumerate}[label=\textnormal{(M)}]
\item\label{M} For $a,b\in\As$, $a \cvr b$ if and only if $(a,b)\in R^{\As}$ for some unique binary relation $R$.
\end{enumerate}
For each $k>0$, the full subcategory $\RMk(\sg)$ of $\RM(\sg)$ consisting of the tree-ordered structures of height at most $k$ is isomorphic to the category of coalgebras for the modal comonad $\Mk$ on $\CSstar$~\cite[Theorem~9.5]{AS2021}. When equipped with the restriction of the factorisation system on $\R(\sg)$, the category $\RM(\sg)$ is arboreal and its paths are those objects in which the order is a finite chain. Likewise for $\RMk(\sg)$.
\end{example}

It follows from the definition of path that, for any object $X$ of an arboreal category, the poset $\Path{X}$ is a tree; in fact, a non-empty tree. Crucially, this assignment extends to a functor into the category of trees (for a proof, see~\cite[Theorem~3.11]{AR2022}):
\begin{theorem}
Let $\C$ be an arboreal category.
The assignment $f\mapsto \Path{f}$ in equation~\eqref{eq:Path-functor} induces a functor $\Path\colon \C\to\T$ into the category of trees.
\end{theorem}

Finally, recall from~\cite[Lemma~3.15, Proposition~5.6 and Remark~5.7]{AR2022} the following properties of paths and posets of embeddings. 
\begin{lemma}\label{l:arboreal:properties}
The following statements hold in any arboreal category $\C$:
\begin{enumerate}[label=(\alph*)]
\item\label{at-most-one-emb} Between any two paths there is at most one embedding.
\item\label{SX-complete-lattice} For all objects $X$ of $\C$, the poset $\Emb{X}$ of its $\M$-subobjects is a complete lattice.\footnote{In fact, $\Emb{X}$ is a \emph{perfect} lattice, \cf \cite{Raney1952} or~\cite{DP2002}.}
\item\label{join-irred} Let $X$ be an object of $\C$ and let $\U\subseteq \Path{X}$ be a non-empty subset. A path embedding $m\in\Path{X}$ is below $\bigvee{\U}\in\Emb{X}$ if, and only if, it is below some element~of~$\U$.
\end{enumerate}
\end{lemma}

If it exists, the unique embedding between paths $P,Q$ in an arboreal category is denoted by 
\[
!_{P,Q}\colon P\emb Q.
\] 
If no confusion arises, we simply write $!\colon P\emb Q$. 

\subsection{Bisimilarity and back-and-forth systems}
Throughout this section, we fix an arbitrary arboreal category $\C$.
A morphism $f\colon X\to Y$ in $\C$ is said to be \emph{open} if it satisfies the following path-lifting property: Given any commutative square
\[\begin{tikzcd}
P \arrow[rightarrowtail]{r} \arrow[rightarrowtail]{d} & X \arrow{d}{f} \arrow[dashed,leftarrow]{dl} \\
Q \arrow[rightarrowtail]{r} & Y
\end{tikzcd}\]
with $P,Q$ paths, there is an arrow $Q\to X$ making the two triangles commute. (If such an arrow exists, it is automatically an embedding.)
Further, $f$ is a \emph{pathwise embedding} if, for all path embeddings $m\colon P\emb X$, the composite $f\circ m\colon P \to Y$ is a path embedding.

Combining these notions, we can define a bisimilarity relation between objects of $\C$:

\begin{definition}
Two objects $X,Y$ of $\C$ are said to be \emph{bisimilar} if there exist an object $Z$ of $\C$ and  a span of open pathwise embeddings $X\leftarrow Z \rightarrow Y$.
\end{definition}

\begin{remark}
The definition of open morphism given above is a refinement of the one introduced in~\cite{JNW1993} (\cf \cite[\S 4.1]{AR2022} for a discussion of the relation between these notions), which is a special case of the (axiomatic) concept of open map in toposes~\cite{JM1994}. 
\end{remark}

As we shall see next, if $\C$ has binary products the bisimilarity relation can be characterised in terms of back-and-forth systems. Let $X,Y$ be objects of $\C$. Given $m\in\Path{X}$ and $n\in\Path{Y}$, we write $\br{m,n}$ to indicate that $\dom(m)\cong \dom(n)$. 
Intuitively, the pair $\br{m,n}$ encodes a partial isomorphisms between $X$ and $Y$ ``of shape $P$'', with $P$ a path. 

\begin{definition}\label{def:back-and-forth}
A \emph{back-and-forth system} between objects $X$ and $Y$ of $\C$ is a set 
\[
\B=\{\br{m_i,n_i}\mid m_i\in\Path{X}, \, n_i\in\Path{Y}, \, i\in I\}
\] 
satisfying the following conditions:
\begin{enumerate}[label=(\roman*)]
\item\label{initial} $\br{\bot_X,\bot_Y}\in\B$, where $\bot_X,\bot_Y$ are the roots of $\Path{X}$ and $\Path{Y}$, respectively.
\item\label{forth} If $\br{m,n}\in\B$ and $m'\in\Path{X}$ are such that $m\cvr m'$, there exists $n'\in\Path{Y}$ satisfying $n\cvr n'$ and $\br{m',n'}\in\B$.
\item\label{back} If $\br{m,n}\in\B$ and $n'\in\Path{Y}$ are such that $n\cvr n'$, there exists $m'\in\Path{X}$ satisfying $m\cvr m'$ and $\br{m',n'}\in\B$.
\end{enumerate}
Two objects $X$ and $Y$ of $\C$ are said to be \emph{back-and-forth equivalent} if there exists a back-and-forth system between them.
\end{definition}

For a proof of the following result, see~\cite[Theorem~6.4]{AR2022}.

\begin{theorem}\label{th:bisimilar-iff-strong-back-forth}
Let $X,Y$ be objects of an arboreal category admitting a product. Then $X$ and $Y$ are bisimilar if, and only if, they are back-and-forth equivalent.
\end{theorem}

The existence of a back-and-forth system between~$X$ and~$Y$ can be equivalently described in terms of the existence of a Duplicator winning strategy in a two-player game played between~$\Path{X}$ and~$\Path{Y}$~\cite[\S 6.2]{AR2022}. Since winning strategies can be composed to yield again a winning strategy, in any arboreal category with binary products the bisimilarity relation is transitive, hence an equivalence relation.

\subsection{Resource-indexed arboreal adjunctions}\label{s:prelim-resource-ind-arb-adj}
Let $\C$ be an arboreal category, with full subcategory of paths $\Cp$. We say that $\C$ is \emph{resource-indexed} if for all positive integers $k$ there is a full subcategory $\Cp^k$ of $\Cp$ closed under embeddings\footnote{\label{fn:closure-emb}That is, for any embedding $P\emb Q$ in $\C$ with $P,Q$ paths, if $Q\in \Cp^k$ then also $P\in \Cp^k$. We shall further assume that each category $\Cp^k$ contains the initial object of $\C$.} with
\[ \Cp^1 \hookrightarrow \Cp^2 \hookrightarrow \Cp^3 \hookrightarrow \cdots \]
This induces a corresponding tower of full subcategories $\C_k$ of $\C$, with the objects of $\C_k$ those whose cocone of path embeddings with domain in $\Cp^k$ is a colimit cocone in $\C$.
It turns out that each category $\C_k$ is arboreal. Furthermore, the paths in $\C_k$ are precisely the objects of $\Cp^k$, \ie $(\C_k)_p = \Cp^k$. \Cf \cite[Proposition~7.7]{AR2022} and its proof.

\begin{example}\label{ex:resource-ind-arb-cat}
Consider the arboreal category $\RT(\sg)$ from Example~\ref{ex:RE}. This can be regarded as a resource-indexed arboreal category by taking as $\Cp^k$ the full subcategory of $\RT(\sg)$ consisting of the objects in which the order is a finite chain of cardinality $\leq k$. The generated subcategory $\C_k$ then coincides with $\RTk(\sg)$. 

Similar reasoning shows that the arboreal category $\RM(\sg)$ from Example~\ref{ex:RM} can also be regarded as a resource-indexed arboreal category.
\end{example}

\begin{definition}
Let $\{\C_k\}$ be a resource-indexed arboreal category and let $\E$ be a category. A \emph{resource-indexed arboreal adjunction} between $\E$ and $\C$ is an indexed family of adjunctions
\[ \begin{tikzcd}
\C_k \arrow[r, bend left=25, ""{name=U, below}, "L_k"{above}]
\arrow[r, leftarrow, bend right=25, ""{name=D}, "R_k"{below}]
& \E.
\arrow[phantom, "\textnormal{\footnotesize{$\bot$}}", from=U, to=D] 
\end{tikzcd}
\]
A \emph{resource-indexed arboreal cover} of $\E$ by $\C$ is a resource-indexed arboreal adjunction between $\E$ and $\C$ such that all adjunctions $L_k\dashv R_k$ are comonadic, \ie for all $k>0$ the comparison functor from $\C_k$ to the category of Eilenberg-Moore coalgebras for the comonad $G_k\coloneqq L_k R_k$ is an isomorphism.
\end{definition}
\begin{example}\label{ex:res-ind-arb-cover}
Let $\sg$ be a relational vocabulary and let $\E=\CS$. Consider the resource-indexed arboreal category $\RT(\sg)$ in Example~\ref{ex:resource-ind-arb-cat}. For each $k > 0$, there is a forgetful functor
\[ 
\LE_k\colon \RTk(\sg) \to \CS
\]
which forgets the forest order. This functor is comonadic. The right adjoint $\RE_k$ sends a $\sg$-structure $\As$ to $\Ek(\As)$ equipped with the prefix order, and the comonad induced by this adjunction coincides with the $k$-round Ehrenfeucht-\Fraisse~comonad $\Ek$. This gives rise to a resource-indexed arboreal cover of $\CS$ by $\RT(\sg)$.

Similarly, if $\sg$ is a modal vocabulary, there is a resource-indexed arboreal cover of $\CSstar$ by $\RM(\sg)$ such that each adjunction $L^M_k \dashv R^M_k$ induces the $k$-round modal comonad $\Mk$.
\end{example}

\begin{example}\label{ex:res-ind-arb-adj}
To deal with the equality symbol in the logic, we consider resource-indexed arboreal adjunctions constructed as follows.
Let 
\[\sg^I\coloneqq \sg\cup \{I\} 
\]
be the vocabulary obtained by adding a fresh binary relation symbol~$I$ to~$\sg$. 
Any $\sg$-structure can be expanded to a $\sg^I$-structure by interpreting $I$ as the identity relation. This yields a fully faithful functor $J\colon \CS\to \CSplus$. The functor $J$ has a left adjoint $H\colon \CSplus\to \CS$ which sends a $\sg^I$-structure $\As$ to the quotient of the $\sg$-reduct of $\As$ by the equivalence relation generated by $I^{\As}$ \cite[Lemma~25]{DJR2021}. We can then compose the adjunction $H\dashv J$ with \eg the Ehrenfeucht-\Fraisse~resource-indexed arboreal cover of $\CSplus$ by $\RT(\sg^I)$ from Example~\ref{ex:res-ind-arb-cover}.
\begin{equation*}
\begin{tikzcd}
{\RTk(\sg^I)} \arrow[r, bend left=25, ""{name=U, below}, "\LE_k"{above}]
\arrow[r, leftarrow, bend right=25, ""{name=D}, "\RE_k"{below}]
& {\CSplus} \arrow[r, bend left=25, ""{name=U', below}, "H"{above}]
\arrow[r, leftarrow, bend right=25, ""{name=D'}, "J"{below}] & {\CS}
\arrow[phantom, "\textnormal{\footnotesize{$\bot$}}", from=U, to=D] 
\arrow[phantom, "\textnormal{\footnotesize{$\bot$}}", from=U', to=D'] 
\end{tikzcd}
\end{equation*}
The composite adjunctions $H \LE_k \dashv \RE_k J$, which are not comonadic, yield the \emph{Ehrenfeucht-\Fraisse~resource-indexed arboreal adjunction} between $\CS$ and $\RT(\sg^I)$.
\end{example}
Crucially, a resource-indexed arboreal adjunction between $\E$ and $\C$ can be used to define several resource-indexed relations between objects of $\E$:
\begin{definition}\label{def:resource-indexed-relations}
Consider a resource-indexed arboreal adjunction between $\E$ and $\C$, with adjunctions $L_k \dashv R_k$, and any two objects $a,b$ of $\E$. For all $k>0$, we define:
\begin{itemize}
    \item $a \rightarrow_k^{\C} b$ if there exists a morphism $R_k a \to R_k b$ in $\C_k$.
    \item $a \eqbCk b$ if $R_k a$ and $R_k b$ are bisimilar in $\C_k$.
    \item $a \eqcCk b$ if $R_k a$ and $R_k b$ are isomorphic in $\C_k$.
\end{itemize}
\end{definition}
Further, we write $\eqaCk$ for the symmetrisation of the preorder $\rightarrow_k^{\C}$. There are inclusions
\[
{\eqcCk} \ \subseteq \ {\eqbCk} \ \subseteq \ {\eqaCk}.
\]
The first inclusion is trivial, the second follows from \cite[Lemma~6.20]{AR2022}. For a proof of the following easy observation, see~\cite[Lemma~7.20]{AR2022}.

\begin{lemma}\label{l:equiv-rel-properties}
Consider a resource-indexed arboreal adjunction between $\E$ and~$\C$, with adjunctions $L_k \dashv R_k$. The following hold for all $a,b\in \E$ and all $k > 0$:
\begin{enumerate}[label=(\alph*)]
\item\label{hom-k-hom} If there exists a morphism $a\to b$ in $\E$ then $a \rightarrow_k^{\C} b$.
\item\label{k-equiv} $a\eqaCk L_k R_k a$.
\end{enumerate}
\end{lemma}

To conclude, we recall how the relations in Definition~\ref{def:resource-indexed-relations} capture, in our running examples, preservation of the logics introduced at the beginning of Section~\ref{s:prelim-game-comonads}.
Given a set of sentences (or modal formulas) $\LL$, let $\Rrightarrow^\LL$ be the preorder on (pointed) $\sg$-structures given by
\[
\As \Rrightarrow^{\LL} \Bs \ \ \Longleftrightarrow \ \ \forall \phi\in\LL. \, (\As\vDash \phi \, \Rightarrow \, \Bs\vDash \phi).
\]
The equivalence relation $\equiv^\LL$ is the symmetrisation of $\Rrightarrow^\LL$.
\begin{example}\label{logical-rel-EF}
Let $\sg$ be a finite relational vocabulary.
Consider the Ehrenfeucht-\Fraisse~resource-indexed arboreal adjunction between $\CS$ and $\RT(\sg^I)$ in Example~\ref{ex:res-ind-arb-adj}, and write $\rightarrow_k^{E}$ and $\leftrightarrow_k^E$ for the relations on $\CS$ induced according to Definition~\ref{def:resource-indexed-relations}. 
For all $\sg$-structures $\As,\Bs$ and all $k>0$, we have
\[
\As \rightarrow_k^{E} \Bs \ \Longleftrightarrow \ \As \Rrightarrow^{\EFO_k} \Bs
\]
and 
\[
\As \leftrightarrow_k^{E} \Bs \ \Longleftrightarrow \ \As \equiv^{\FO_k} \Bs.
\]
\Cf \cite[Theorems~3.2 and~5.1]{AS2021} and~\cite[Theorem~10.5]{AS2021}, respectively. We also mention that $\cong_k^E$ coincides with equivalence in the extension of $\FO_k$ with \emph{counting quantifiers} \cite[Theorem~5.3(2)]{AS2021}, although we shall not need this fact.
\end{example}

\begin{example}\label{logical-rel-modal}
Suppose $\sg$ is a finite modal vocabulary and consider the relations $\rightarrow_k^{M}$ and $\cong_k^M$ on $\CSstar$ induced by the modal resource-indexed arboreal cover of $\CSstar$ by $\RMk(\sg)$ in Example~\ref{ex:res-ind-arb-cover}. For all pointed Kripke structures $(\As,a),(\Bs,b)$ and all $k>0$, we have
\[
(\As,a)\rightarrow_k^{M} (\Bs,b) \ \Longleftrightarrow \ (\As,a)\Rrightarrow^{\exists^+\ML_k} (\Bs,b),
\]
see \cite[Theorem~9]{DBLP:conf/csl/AbramskyS18}. 
Furthermore, 
\[
(\As,a)\cong_k^M (\Bs,b) \ \Longrightarrow \ (\As,a)\equiv^{\ML_k(\#)} (\Bs,b),
\]
\cf \cite[Proposition~15]{DBLP:conf/csl/AbramskyS18} and~\cite[Proposition~3.6]{deRijke2000}. We mention in passing that the relation $\leftrightarrow_k^M$ coincides with equivalence in $\ML_k$ \cite[Theorem~10.13]{AS2021}.
\end{example}

\section{Homomorphism preservation theorems}\label{s:logics-HPTs}
%
%
In this section, we recast the statement of a generic equi-resource homomorphism preservation theorem into a property \textnormal{(HP)}---and its strengthening \textnormal{(HP${}^\#$)}---that a resource-indexed arboreal adjunction may or may not satisfy. Sufficient conditions under which properties \textnormal{(HP)} and \textnormal{(HP${}^\#$)} hold are provided in Section~\ref{s:exploring-the-landscape}.
\subsection{\textnormal{(HP)} and \textnormal{(HP${}^\#$)}}\label{s:HP-HPplus}
Given a first-order sentence $\phi$ in a relational vocabulary~$\sg$, its ``model class'' $\Mod(\phi)$ is the full subcategory of $\CS$ defined by the $\sg$-structures $\As$ such that $\As\vDash\phi$.
To motivate the formulation of properties \textnormal{(HP)} and \textnormal{(HP${}^\#$)}, we recall a well-known characterisation of model classes of sentences in $\FO_k$, \ie first-order sentences of quantifier rank at most $k$, and in its existential positive fragment $\EFO_k$. Since a sentence can only contain finitely many relation symbols, for the purpose of investigating homomorphism preservation theorems we can safely assume that $\sg$ is finite.

For a full subcategory $\D$ of a category $\A$, and a relation $\nabla$ on the class of objects of $\A$, we say that $\D$ is \emph{upwards closed (in $\A$) with respect to $\nabla$} if 
\[
\forall a,b \in \A, \text{ if } a\in\D \text{ and } a \, \nabla \, b \text{ then } b\in\D.
\] 
If $\nabla$ is an equivalence relation and the latter condition is satisfied, we say that $\D$ is \emph{saturated under $\nabla$}.
The following lemma follows from the fact that, for all $k \geq 0$, there are finitely many sentences in $\FO_k$ up to logical equivalence. \Cf \eg \cite[Lemma~3.13]{Libkin2004}.
\begin{lemma}\label{concrete-charact-log-eq}
The following hold for all $k\geq 0$ and all full subcategories $\D$ of $\CS$:
\begin{enumerate}[label=(\alph*)]
\item\label{synt-free} $\D=\Mod(\phi)$ for some $\phi\in\FO_k$ if, and only if, $\D$ is saturated under $\equiv^{\FO_k}$.
\item\label{synt-free-EP} $\D=\Mod(\psi)$ for some $\psi\in\EFO_k$ if, and only if, $\D$ is upwards closed with respect to $\Rrightarrow^{\EFO_k}$.
\end{enumerate}
\end{lemma}

\begin{remark}\label{rem:finite-fragments}
The previous lemma remains true if $\FO_k$ is replaced with any fragment of first-order logic that is closed under Boolean connectives and contains, up to logical equivalence, finitely many sentences.
\end{remark}

Now, fix a resource-indexed arboreal adjunction between $\E$ and $\C$, with adjunctions 
\[\begin{tikzcd}
\C_k \arrow[r, bend left=25, ""{name=U, below}, "L_k"{above}]
\arrow[r, leftarrow, bend right=25, ""{name=D}, "R_k"{below}]
& \E.
\arrow[phantom, "\textnormal{\footnotesize{$\bot$}}", from=U, to=D] 
\end{tikzcd}\]
Let us say that a full subcategory $\D$ of $\E$ is \emph{closed (in $\E$) under morphisms} if, whenever there is an arrow $a\to b$ in $\E$ with $a\in \D$, also $b\in \D$. 
Note that, when $\E=\CS$ and $\D=\Mod(\phi)$ is the model class of some sentence $\phi$, the category $\D$ is closed under morphisms precisely when $\phi$ is preserved under homomorphisms. 

Consider the following statement, where $\prCk$ and $\eqbCk$ are the relations on the objects of $\E$ induced by the resource-indexed arboreal adjunction as in Definition~\ref{def:resource-indexed-relations}:
\begin{enumerate}[label=\textnormal{(HP)}]
\item\label{HP-abstract} For any full subcategory $\D$ of $\E$ saturated under $\eqbCk$, $\D$ is closed under morphisms precisely when it is upwards closed with respect to $\prCk$.
\end{enumerate}

Replacing the relation $\eqbCk$ with $\eqcCk$, we obtain a strengthening of \ref{HP-abstract}, namely:
\begin{enumerate}[label=\textnormal{(HP${}^\#$)}]
\item\label{HPplus-abstract} For any full subcategory $\D$ of $\E$ saturated under $\eqcCk$, $\D$ is closed under morphisms precisely when it is upwards closed with respect to $\prCk$.
\end{enumerate}
Just recall that ${\eqcCk} \subseteq {\eqbCk}$, and so any full subcategory $\D$ saturated under $\eqbCk$ is also saturated under $\eqcCk$. Thus, \ref{HPplus-abstract} entails~\ref{HP-abstract}.
\begin{remark}\label{r:easy-dir-HPTs}
By Lemma~\ref{l:equiv-rel-properties}\ref{hom-k-hom}, any full subcategory of $\E$ that is upwards closed with respect to $\prCk$ is closed under morphisms. Hence, the right-to-left implications in~\ref{HP-abstract} and~\ref{HPplus-abstract} are always satisfied.
\end{remark}

In view of Example~\ref{logical-rel-EF} and Lemma~\ref{concrete-charact-log-eq}, for the Ehrenfeucht-\Fraisse~resource-indexed arboreal adjunction between $\CS$ and $\RT(\sg^I)$, property \ref{HP-abstract} coincides with Rossman's equirank homomorphism preservation theorem (Theorem~\ref{th:equirank-HPT}).

In Section~\ref{s:axiomatic} we will prove that \ref{HP-abstract} holds for any resource-indexed arboreal adjunction satisfying appropriate properties (see Corollary~\ref{cor:HPT-axiomatic}), which are satisfied in particular by the Ehrenfeucht-\Fraisse~resource-indexed arboreal adjunction.

\section{Exploring the landscape: tame and not-so-tame}
\label{s:exploring-the-landscape}
We shall identify, in Section~\ref{ss:tame}, a class of ``tame'' resource-indexed arboreal adjunctions, namely those satisfying the \emph{bisimilar companion property}, for which \textnormal{(HP)} always holds. 
As an application, we derive equi-resource homomorphism preservation theorems for graded modal logic and guarded first-order logics.

In the absence of the bisimilar companion property, one may try to ``force'' it; this leads us, in Section~\ref{ss:extendability}, to the notion of \emph{extendability}, inspired by the work of Rossman~\cite{Rossman2008}. 
Finally, in Section~\ref{s:relativisation}, we provide simple sufficient conditions under which properties \textnormal{(HP)} and \textnormal{(HP${}^\#$)} admit a relativisation to a full subcategory.

\subsection{Tame: bisimilar companion property and idempotency}\label{ss:tame}
For all $k>0$, write $G_k\coloneqq L_k R_k$ for the comonad on $\E$ induced by the adjunction ${L_k\dashv R_k \colon \E \to \C_k}$. 
\begin{definition}
A resource-indexed arboreal adjunction between $\E$ and $\C$, with induced comonads $G_k$, has the \emph{bisimilar companion property} if $a \eqbCk G_k a$ for all $a\,{\in}\, \E$ and $k > 0$.
\end{definition}

\begin{proposition}\label{p:HPT-tame}
\ref{HP-abstract} holds for any resource-indexed arboreal adjunction between $\E$ and $\C$ satisfying the bisimilar companion property.
\end{proposition}
\begin{proof}
For the left-to-right implication in~\ref{HP-abstract}, let $\D$ be a full subcategory of $\E$ closed under morphisms and saturated under~$\eqbCk$. Suppose that $a\prCk b$ for objects $a,b$ of $\E$. By definition, this means that there is an arrow $R_k a \to R_k b$ and so, as $L_k$ is left adjoint to $R_k$, there is an arrow $G_k a \to b$. Using the bisimilar companion property, we get
\[
a \, \eqbCk \, G_k a\,  \to \, b.
\]
Therefore, if $a\in \D$ then also $b\in \D$. That is, $\D$ is upwards closed with respect to $\prCk$.

The converse direction follows from Remark~\ref{r:easy-dir-HPTs}.
\end{proof}

In order to establish a similar result for property \ref{HPplus-abstract}, recall that a comonad $G$ is \emph{idempotent} if its comultiplication $G \Rightarrow G G$ is a natural isomorphism.
\begin{definition}
A resource-indexed arboreal adjunction between $\E$ and $\C$ is \emph{idempotent} if so are the induced comonads $G_k$, for all $k>0$.
\end{definition}

\begin{proposition}\label{p:HPT-graded}
\ref{HPplus-abstract} holds for any idempotent resource-indexed arboreal adjunction between $\E$ and $\C$.
\end{proposition}
\begin{proof}
Recall that $G_k$ is idempotent if, and only if, $\eta R_k$ is a natural isomorphism, where $\eta$ is the unit of the adjunction $L_k \dashv R_k$. In particular, for any $a\in \E$, the component of $\eta R_k$ at $a$ yields an isomorphism $R_k a \cong R_k G_k a$ in $\C$. Hence, $a \eqcCk G_k a$ for all $a\in\E$.
Reasoning as in the proof of Proposition~\ref{p:HPT-tame}, it is easy to see that~\ref{HPplus-abstract} holds.
\end{proof}

\begin{remark}
Consider an idempotent resource-indexed arboreal adjunction between $\E$ and $\C$ with induced comonads $G_k$ on $\E$. The previous proof shows that, for all $a\in \E$ and $k>0$, we have $a \eqcCk G_k a$. A fortiori, $a \eqbCk G_k a$. Therefore, any idempotent resource-indexed arboreal adjunction satisfies the bisimilar companion property.
\end{remark}

Next, we show how Propositions~\ref{p:HPT-tame} and~\ref{p:HPT-graded} can be exploited to obtain equi-resource homomorphism preservation theorems for (graded) modal logic and guarded first-order logics. Relativisations of these results to subclasses of structures, \eg to the class of all finite structures, are discussed in Section~\ref{s:relativisation}.

\subsection*{Graded modal logic}
Let $\sigma$ be a finite modal vocabulary. As observed in~\cite[\S 9.3]{AS2021}, the modal comonads $\Mk$ on $\CSstar$ are idempotent, hence so is the modal resource-indexed arboreal cover of $\CSstar$ by $\RMk(\sg)$. This corresponds to the fact that a tree-ordered Kripke structure is isomorphic to its tree unravelling. 
Thus, Proposition~\ref{p:HPT-graded} entails the following \emph{equidepth homomorphism preservation theorem} for graded modal formulas (\iec modal formulas that possibly contain graded modalities):

\begin{theorem}\label{th:hpt-graded-modal-logic}
The following statements are equivalent for any graded modal formula~$\phi$ of modal depth at most $k$ in a modal vocabulary:
\begin{enumerate}
\item $\phi$ is preserved under homomorphisms between pointed Kripke structures.
\item $\phi$ is logically equivalent to an existential positive modal formula of modal depth at most $k$.
\end{enumerate}
\end{theorem}
\begin{proof}
Fix a graded modal formula $\phi\in\ML_k(\#)$. Since a single modal formula contains only finitely many modalities and propositional variables, we can assume without loss of generality that $\phi$ is a formula in a finite modal vocabulary. By Example~\ref{logical-rel-modal} we have
\[
{\rightarrow_k^{M}} = {\Rrightarrow^{\exists^+\ML_k}} \ \text{ and } \ {\cong_k^M} \subseteq {\equiv^{\ML_k(\#)}}.
\]
In particular, the latter inclusion entails that the full subcategory $\Mod(\phi)$ of $\CSstar$ is saturated under $\cong_k^M$. 
As the modal resource-indexed arboreal cover is idempotent, Proposition~\ref{p:HPT-graded} implies that $\Mod(\phi)$ is closed under morphisms if, and only if, it is upwards closed with respect to $\to_k^M$. Note that $\Mod(\phi)$ is closed under morphisms precisely when $\phi$ is preserved under homomorphisms between pointed Kripke structures. On the other hand, the equality ${\rightarrow_k^{M}} = {\Rrightarrow^{\exists^+\ML_k}}$ implies that $\Mod(\phi)$ is upwards closed with respect to $\to_k^M$ if, and only if, $\Mod(\phi)=\Mod(\psi)$ for some $\psi\in \exists^+\ML_k$ (this is akin to Lemma~\ref{concrete-charact-log-eq}\ref{synt-free-EP} and hinges on the fact that $\exists^+\ML_k$ contains finitely many formulas up to logical equivalence). Thus the statement follows.
\end{proof}

\begin{remark}
Forgetting about both graded modalities and modal depth, Theorem~\ref{th:hpt-graded-modal-logic} implies that a modal formula is preserved under homomorphisms if, and only if, it is equivalent to an existential positive modal formula. This improves the well known result that a modal formula is preserved under simulations precisely when it is equivalent to an existential positive modal formula (see \eg \cite[Theorem~2.78]{blackburn2002modal}). 
\end{remark}

\subsection*{Guarded fragments of first-order logic} 
The study of guarded fragments of first-order logic was initiated by Andr\'eka, van Benthem and N\'emeti in~\cite{HNvB1998} to analyse, and extend to the first-order setting, the good algorithmic and model-theoretic properties of modal logic. \emph{Guarded formulas} (over a relational vocabulary $\sg$) are defined by structural induction, starting from atomic formulas and applying Boolean connectives and the following restricted forms of quantification: if $\phi(\o{x},\o{y})$ is a guarded formula, then so are
\[
\exists \o{x}. \, G(\o{x},\o{y}) \wedge \phi(\o{x},\o{y}) \ \text{ and } \ \forall \o{x}. \, G(\o{x},\o{y}) \to \phi(\o{x},\o{y})
\]
where $G$ is a so-called \emph{guard}. The (syntactic) conditions imposed on guards determine different guarded fragments of first-order logic. We shall consider the following two:
\begin{itemize}
\item \emph{Atom guarded:} $G(\o{x},\o{y})$ is an atomic formula in which all variables in $\o{x},\o{y}$ occur.
\item \emph{Loosely guarded:} $G(\o{x},\o{y})$ is a conjunction of atomic formulas such that each pair of variables, one in $\o{x}$ and the other in $\o{x},\o{y}$, occurs in one of the conjuncts.
\end{itemize} 

The atom guarded fragment of first-order logic was introduced in~\cite{HNvB1998} under the name of F2 (``Fragment~$2$''), whereas the loosely guarded fragment was defined by van Benthem in~\cite{vBpieces}. The atom guarded fragment can be regarded as an extension of modal logic, in the sense that the standard translation of the latter is contained in the former. In turn, the loosely guarded fragment extends the atomic one and can express \eg (the translation of) the \emph{Until} modality in temporal logic, \cf \cite[p.~9]{vBpieces}.

For each notion of guarding $\g$ (atomic or loose), denote by 
\[
\g\FO^n \ \text{ and } \ \exists^+\g\FO^n,
\] 
respectively, the $n$-variable $\g$-guarded fragment of first-order logic and its existential positive fragment. 
In~\cite{Guarded2021}, \emph{guarded comonads} $\mathbb{G}_n^{\g}$ on $\CS$ are defined for all $n>0$. The associated categories of Eilenberg-Moore coalgebras are arboreal and induce the \emph{$\g$-guarded resource-indexed arboreal cover} of $\CS$ with resource parameter $n$. For an explicit description of the resource-indexed arboreal category in question, \cf \cite[\S IV]{Guarded2021}.

Assume the vocabulary $\sg$ is finite and let $\rightarrow_n^{\g}$ and $\leftrightarrow_n^{\g}$ be the resource-indexed relations on $\CS$ induced by the $\g$-guarded resource-indexed arboreal cover. It follows from \cite[Theorems~III.4 and~V.2]{Guarded2021} that, for all $\sg$-structures $\As,\Bs$ and all $n>0$, 
\[
\As \rightarrow_n^{\g} \Bs \ \Longleftrightarrow \ \As \Rrightarrow^{\exists^+\g\FO^n} \Bs
\]
and 
\[
\As \leftrightarrow_n^{\g} \Bs \ \Longleftrightarrow \ \As \equiv^{\g\FO^n} \Bs.
\]
To obtain an analogue of Lemma~\ref{concrete-charact-log-eq}, we consider finite fragments of $\g\FO^n$ by stratifying in terms of \emph{guarded-quantifier rank} (\cf Remark~\ref{rem:finite-fragments}).
Note that, as guarded quantifiers bound \emph{tuples} of variables, rather than single variables, the guarded-quantifier rank of a guarded formula is typically lower than its ordinary quantifier rank. Nevertheless, for all $k\geq 0$, the fragment $\g\FO^n_k$ of $\g\FO^n$ consisting of those sentences with guarded-quantifier rank at most $k$ contains finitely many sentences up to logical equivalence.

This stratification can be modelled in terms of comonads $\mathbb{G}_{n,k}^{\g}$ on $\CS$, for all $n,k>0$, as explained in \cite[\S VII]{Guarded2021}. Fixing $n$ and letting $k$ vary, we obtain an \emph{$n$-variable $\g$-guarded resource-indexed arboreal cover} of $\CS$, with resource parameter $k$. The induced relations $\rightarrow_{n,k}^{\g}$ and $\leftrightarrow_{n,k}^{\g}$ on $\CS$ coincide, respectively, with preservation of $\exists^+\g\FO^n_k$ and equivalence in $\g\FO^n_k$. Thus, for any full subcategory $\D$ of $\CS$:
\begin{itemize}
\item $\D=\Mod(\phi)$ for some $\phi\in\g\FO^n_k$ if, and only if, $\D$ is saturated under $\leftrightarrow_{n,k}^{\g}$.
\item  $\D=\Mod(\psi)$ for some $\psi\in\exists^+\g\FO^n_k$ if, and only if, $\D$ is upwards closed with respect to $\rightarrow_{n,k}^{\g}$.
\end{itemize}

As observed in \cite[\S 6.1]{Hybrid2022}, the $\g$-guarded resource-indexed arboreal cover of $\CS$ satisfies the bisimilar companion property, and so does the $n$-variable $\g$-guarded resource-indexed arboreal cover for all $n>0$. Therefore, Proposition~\ref{p:HPT-tame} implies the following \emph{equirank-variable homomorphism preservation theorem} for guarded logics:
\begin{theorem}\label{th:hpt-guarded}
Let $\g$ be a notion of guarding (either atom or loose).
The following statements are equivalent for any $\g$-guarded sentence~$\phi$ in $n$ variables of guarded-quantifier rank at most $k$ in a relational vocabulary:
\begin{enumerate}
\item $\phi$ is preserved under homomorphisms.
\item $\phi$ is logically equivalent to an existential positive $\g$-guarded sentence in $n$ variables of guarded-quantifier rank at most $k$.
\end{enumerate}
\end{theorem}

\subsection{Not-so-tame: extendability}\label{ss:extendability}
A resource-indexed arboreal adjunction may fail to satisfy the bisimilar companion property, in which case Proposition~\ref{p:HPT-tame} does not apply. This is the case \eg for the Ehrenfeucht-\Fraisse~resource-indexed arboreal adjunction:
\begin{example}\label{ex:EF-bcp-fails}
The Ehrenfeucht-\Fraisse~resource-indexed arboreal adjunction between $\CS$ and $\RT(\sg^I)$ does not have the bisimilar companion property. Suppose that $\sg=\{R\}$ consists of a single binary relation symbol and let $\As$ be the $\sg$-structure with underlying set $\{a,b\}$ satisfying $R^{\As}=\{(a,b),(b,a)\}$. In view of Examples~\ref{ex:res-ind-arb-adj} and~\ref{logical-rel-EF}, it suffices to find $k>0$ and a first-order sentence $\phi$ of quantifier rank $\leq k$ such that 
\[
\As \vDash \phi \ \text{ and } \ H\Ek J\As \not\vDash \phi.
\]
Let $\phi$ be the sentence
$
\forall x \forall y \ (x\neq y \Rightarrow xRy)
$
of quantifier rank $2$ stating that any two distinct elements are $R$-related. Then $\phi$ is satisfied by $\As$ but not by $H\Ek J\As$, because the sequences $[a]$ and $[b]$ are not $R$-related in $H\Ek J\As$. This shows that the bisimilar companion property fails for all $k\geq 2$.
\end{example}

When the bisimilar companion property fails, \ie $a \not\eqbCk G_k a$ for some $a\in \E$ and $k > 0$, we may attempt to force it by finding appropriate extensions $a^*$ and $(G_k a)^*$ of $a$ and $G_k a$, respectively, satisfying $a^* \eqbCk (G_k a)^*$. This motivates the notion of \emph{$k$-extendability} (see Definition~\ref{d:k-extendable} below), inspired by the work of Rossman~\cite{Rossman2008} and its categorical interpretation in~\cite{ABRAMSKY2020}.

To start with, we introduce the following notations. 
Given objects $a,b$ of a category~$\A$, we write $a\to b$ to denote the existence of an arrow from $a$ to $b$. Further, we write $a\rl b$ to indicate that $a\to b$ and $b\to a$, \ie $a$ and $b$ are \emph{homomorphically equivalent}. 
This applies in particular to coslice categories.
Recall that, for any $c\in\A$, the \emph{coslice category} $c/{\A}$ (also known as \emph{under category}) has as objects the arrows in $\A$ with domain~$c$. For any two objects $m\colon c\to a$ and $n\colon c\to b$ of $c/{\A}$, an arrow $f\colon m\to n$ in $c/{\A}$ is a morphism $f\colon a\to b$ in $\A$ such that $f\circ m = n$. Hence, $m\rl n$ in $c/{\A}$ precisely when there are arrows $f\colon a\to b$ and $g\colon b\to a$ in $\A$ satisfying $f\circ m = n$ and $g\circ n = m$. We shall represent this situation by means of the following diagram:
\[\begin{tikzcd}[column sep=2em]
{} & c \arrow{dl}[swap]{m} \arrow{dr}{n} & {} \\
a \arrow[yshift=3pt]{rr}{f} & & b \arrow[yshift=-3pt]{ll}{g}
\end{tikzcd}\]

\begin{remark}\label{rem:sections-coslice}
Note that $m\rl n$ in $c/{\A}$ whenever there is a section $f\colon a \to b$ in $\A$ satisfying $f\circ m = n$.
Just observe that the left inverse $f^{-1}$ of $f$ satisfies
\[
f^{-1} \circ n = f^{-1} \circ f\circ m = m
\]
and so $n\to m$. Further, $f\circ m = n$ entails $m\to n$. Hence, $m\rl n$.
\end{remark}

Now, let $a$ be an object of an arboreal category $\C$ and let $m\colon P\emb a$ be a path embedding. 
As $\Emb{a}$ is a complete lattice by Lemma~\ref{l:arboreal:properties}\ref{SX-complete-lattice}, the supremum $\bigvee{{\uparrow}m}$ in $\Emb{a}$ of all path embeddings above $m$ exists and we shall denote it by
\[
\inc{m}\colon \Sg{m}\emb a.
\]
Clearly, $m\leq \inc{m}$ in $\Emb{a}$, and so there is a path embedding
\[
\co{m}\colon P\emb \Sg{m}
\]
satisfying $\inc{m} \circ \co{m} = m$. (Note that $\inc{m}$ is well defined only up to isomorphism in the coslice category ${\C}/a$, but as usual we work with representatives for isomorphism classes.)
\begin{remark}
To provide an intuition for the previous definition, let us say for the sake of this remark that a path embedding $m\colon P\emb a$ is ``dense in $a$'' if all elements of $\Path{a}$ are comparable with $m$. Then Lemma~\ref{l:corestriction-properties}\ref{comparable-subtree} below implies that $\inc{m}\colon \Sg{m}\emb a$ is the largest $\M$-subobject of $a$ in which $m$ is dense.
\end{remark}
\begin{lemma}\label{l:corestriction-properties}
The following statements hold for all path embeddings ${m\colon P\emb a}$:
\begin{enumerate}[label=(\alph*)]
\item\label{comparable-subtree} $\Path{\Sg{m}}$ is isomorphic to the subtree of $\Path{a}$ consisting of the elements that are comparable with $m$.
\item\label{corestriction-arrows-finer} For all path embeddings $n\colon P\emb b$ and arrows $f\colon a\to b$ such that ${f\circ m = n}$, there is a unique $g\colon \Sg{m}\to \Sg{n}$ making the following diagram commute. 
\[\begin{tikzcd}
\Sg{m} \arrow[rightarrowtail]{d}[swap]{\inc{m}} \arrow[dashed]{r}{g} & \Sg{n} \arrow[rightarrowtail]{d}{\inc{n}} \\
a \arrow{r}{f}  & b
\end{tikzcd}\]
\item\label{corestriction-arrows} For all path embeddings $n\colon P\emb b$, if $m\to n$ then $\co{m}\to \co{n}$.
\end{enumerate}
\end{lemma}
\begin{proof}
\ref{comparable-subtree} The map $\inc{m}\circ - \colon \Emb{\Sg{m}}\to \Emb{a}$ is an order-embedding by Lemma~\ref{l:emb-quo-order-embeddings}, and so its restriction $\Path{\Sg{m}}\to \Path{a}$ is an injective forest morphism. Hence, $\Path{\Sg{m}}$ is isomorphic to the subtree of $\Path{a}$ consisting of those elements that factor through $\inc{m}$, \ie that are below $\bigvee{{\uparrow}m}$. By Lemma~\ref{l:arboreal:properties}\ref{join-irred}, an element of $\Path{a}$ is below $\bigvee{{\uparrow}m}$ precisely when it is below some element of ${\uparrow}m$. In turn, the latter is equivalent to being comparable with~$m$.

\ref{corestriction-arrows-finer} Since $\Sg{m}$ is path-generated, it is the colimit of the canonical cocone $C$ of path embeddings over the small diagram $\Path{\Sg{m}}$ which, by item~\ref{comparable-subtree}, can be identified with the subdiagram of $\Path{a}$ consisting of those elements comparable with $m$. As $\Path{(f\circ \inc{m})}$ is monotone, it sends path embeddings comparable with $m$ to path embeddings comparable with $n$, and so the cocone $\{f\circ \inc{m}\circ p\mid p\in C\}$ factors through $\inc{n}\colon \Sg{n}\emb b$. Hence, there is $g\colon \Sg{m}\to \Sg{n}$ such that $\inc{n} \circ g = f\circ \inc{m}$. Finally, note that if $g'\colon \Sg{m}\to \Sg{n}$ satisfies $\inc{n} \circ g' = f\circ \inc{m}$ then we have $\inc{n} \circ g = \inc{n} \circ g'$, and so $g=g'$ because $\inc{n}$ is monic.

\ref{corestriction-arrows} Suppose there exists $f\colon a\to b$ such that $f\circ m = n$. By item~\ref{corestriction-arrows-finer}, there is $g\colon \Sg{m}\to \Sg{n}$ such that $\inc{n} \circ g = f\circ \inc{m}$. Therefore,
\[
\inc{n}\circ g\circ \co{m} = f\circ \inc{m} \circ \co{m} = f\circ m = n = \inc{n} \circ \co{n}
\]
and so $g\circ \co{m} = \co{n}$ because $\inc{n}$ is a monomorphism.
\end{proof}

\begin{remark}\label{rem:idempotent}
Lemma~\ref{l:corestriction-properties}\ref{corestriction-arrows} entails that $\co{m}\to \co{n}$ whenever $\co{m}\to n$. Just observe that $\co{(\co{m})}$ can be identified with $\co{m}$. 
\end{remark}

\begin{definition}\label{d:k-extendable}
Consider a resource-indexed arboreal adjunction between $\E$ and $\C$, with adjunctions $L_k\dashv R_k$.
 An object $a$ of $\E$ is \emph{$k$-extendable} if it satisfies the following property for all $e\in\E$: For all path embeddings $m\colon P\emb R_k a$ and $n \colon P\emb R_k e$ such that $\co{m} \rl \co{n}$ in the coslice category $P/{\C_k}$ (see the leftmost diagram below),
 
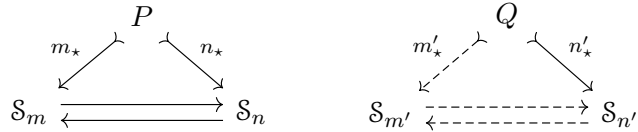
\begin{center}
\begin{tikzcd}[column sep=2em]
{} & P \arrow[rightarrowtail]{dl}[swap]{\co{m}} \arrow[rightarrowtail]{dr}{\co{n}} & {} \\
\Sg{m} \arrow[yshift=3pt]{rr} & & \Sg{n} \arrow[yshift=-3pt]{ll}
\end{tikzcd}
\ \ \ \ \ \ \ 
\begin{tikzcd}[column sep=2em]
{} & Q \arrow[rightarrowtail, dashed]{dl}[swap]{\co{m'}} \arrow[rightarrowtail]{dr}{\co{n'}} & {} \\
\Sg{m'} \arrow[yshift=3pt, dashed]{rr} & & \Sg{n'} \arrow[yshift=-3pt, dashed]{ll}
\end{tikzcd}
\end{center}
if $n'\colon Q\emb R_k e$ is a path embedding satisfying $n\leq n'$ in $\Path{(R_k e)}$, there is a path embedding $m'\colon Q\emb R_k a$ such that $m\leq m'$ in $\Path{(R_k a)}$ and $\co{m'}\rl \co{n'}$ in $Q/{\C_k}$ (as displayed in the rightmost diagram above).
\end{definition}

We shall see in Proposition~\ref{p:extendable} below that, under appropriate assumptions, the $k$-extendability property allows us to upgrade the relation $\eqaCk$ to the finer relation $\eqbCk$.
For the next lemma, recall that a category is \emph{locally finite} if there are finitely many arrows between any two of its objects.
\begin{lemma}\label{l:quotients-inverse}
Let $\C$ be an arboreal category whose full subcategory $\Cp$ consisting of the paths is locally finite. If $f\colon P\epi Q$ and $g\colon Q\epi P$ are quotients between paths, then $f$ and $g$ are inverse to each other.
\end{lemma}
\begin{proof}
The set $M$ of quotients $P\epi P$ is a finite monoid with respect to composition, and it satisfies the right-cancellation law because every quotient is an epimorphism. Hence $M$ is a group, and so $g\circ f\in M$ has an inverse. It follows that $g\circ f$ is an embedding. Because there is at most one embedding between any two paths by Lemma~\ref{l:arboreal:properties}\ref{at-most-one-emb}, $g\circ f = \id_P$. By symmetry, also $f\circ g = \id_Q$.
\end{proof}

\begin{proposition}\label{p:extendable}
Consider a resource-indexed arboreal adjunction between $\E$ and $\C$ such that $\Cp^k$ is locally finite for all $k > 0$. For all $k$-extendable objects $a,b$ of $\E$ admitting a product, we have $a\eqaCk b$ if and only if $a\eqbCk b$.
\end{proposition}
\begin{proof}
Fix an arbitrary $k > 0$ and recall that $\C_k$ is an arboreal category. 
The ``if'' part of the statement follows from the inclusion ${\eqbCk} \subseteq {\eqaCk}$. 

For the ``only if'' part suppose that $a\eqaCk b$, \ie $R_k a$ and $R_k b$ are homomorphically equivalent in $\C_k$. To improve readability, let $X\coloneqq R_k a$ and $Y\coloneqq R_k b$. We must prove that $X$ and $Y$ are bisimilar. 
As $X$ and $Y$ admit a product in $\C_k$, namely the image under $R_k$ of the product of $a$ and $b$ in $\E$, by Theorem~\ref{th:bisimilar-iff-strong-back-forth} it suffices to show that $X$ and $Y$ are back-and-forth equivalent. 

Fix arbitrary morphisms $f\colon X\to Y$ and $g\colon Y\to X$, and let $m$ and $n$ denote generic elements of $\Path{X}$ and $\Path{Y}$, respectively. We claim that
\[
\B\coloneqq\{\br{m,n}\mid \exists \ s\colon \Sg{m}\to \Sg{n}, \, t\colon \Sg{n} \to \Sg{m} \ \text{s.t.} \ \Path{s}(\co{m})=\co{n} \ \text{and} \ \Path{t}(\co{n})=\co{m} \}
\]
is a back-and-forth system between $X$ and $Y$, \ie it satisfies items~\ref{initial}--\ref{back} in Definition~\ref{def:back-and-forth}.
For item~\ref{initial}, let $\bot_X,\bot_Y$ be the roots of $\Path{X}$ and $\Path{Y}$, respectively. Note that $\Sg{\bot_X}$ and $\Sg{\bot_Y}$ can be identified, respectively, with $X$ and $Y$. As $\Path{f}$ and $\Path{g}$ are forest morphisms, $\Path{f}(\bot_X)=\bot_Y$ and $\Path{g}(\bot_Y)=\bot_X$. So, $\br{\bot_X,\bot_Y}\in \B$.

For item~\ref{forth}, suppose $\br{m,n}\in\B$ and let $m'\in\Path{X}$ satisfy $m\cvr m'$. We seek $n'\in\Path{Y}$ such that $n\cvr n'$ and $\br{m',n'}\in\B$. By assumption, there are arrows $s\colon \Sg{m}\to \Sg{n}$ and $t\colon \Sg{n}\to \Sg{m}$ such that $\Path{s}(\co{m})=\co{n}$ and $\Path{t}(\co{n})=\co{m}$. Writing $P\coloneqq \dom(m)$ and $P'\coloneqq \dom(n)$, we have the following diagrams
\begin{center}
\begin{tikzcd}[row sep = 3em]
P \arrow[twoheadrightarrow]{r}{e} \arrow[rightarrowtail]{d}[swap]{\co{m}} & {\cdot} \arrow[rightarrowtail]{d}[description]{\Path{s}(\co{m})} \arrow{r}{\phi} & P' \arrow[rightarrowtail, bend left=30]{dl}{\co{n}} \\
\Sg{m} \arrow{r}{s}& \Sg{n} & 
\end{tikzcd}
\ \ \ \ \ \ \ 
\begin{tikzcd}[row sep = 3em]
P' \arrow[twoheadrightarrow]{r}{e'} \arrow[rightarrowtail]{d}[swap]{\co{n}} & {\cdot} \arrow[rightarrowtail]{d}[description]{\Path{t}(\co{n})} \arrow{r}{\psi} & P \arrow[rightarrowtail, bend left=30]{dl}{\co{m}} \\
\Sg{n} \arrow{r}{t}& \Sg{m} & 
\end{tikzcd}
\end{center}
where $\phi$ and $\psi$ are isomorphisms. By Lemma~\ref{l:quotients-inverse}, $\phi\circ e$ and $\psi\circ e'$ are inverse to each other, thus the left-hand diagram below commutes.
\begin{center}
\begin{tikzcd}[column sep=2em]
{} & P \arrow[rightarrowtail]{dl}[swap]{\co{m}} \arrow[rightarrowtail]{dr}{\co{n}\circ \phi\circ e} & {} \\
\Sg{m} \arrow[yshift=3pt]{rr}{s} & & \Sg{n} \arrow[yshift=-3pt]{ll}{t}
\end{tikzcd}
\ \ \ \ \ \ \ 
\begin{tikzcd}[column sep=2em]
{} & Q \arrow[rightarrowtail]{dl}[swap]{\co{m'}} \arrow[rightarrowtail, dashed]{dr}{\co{n'}} & {} \\
\Sg{m'} \arrow[yshift=3pt, dashed]{rr}{s'} & & \Sg{n'} \arrow[yshift=-3pt, dashed]{ll}{t'}
\end{tikzcd}
\end{center}
Let $Q\coloneqq \dom(m')$. Since $b$ is $k$-extendable, there exist a path embedding $n'\colon Q\emb Y$ such that $n\leq n'$ in $\Path{Y}$, and arrows $s'\colon \Sg{m'}\to \Sg{n'}$ and $t'\colon \Sg{n'}\to \Sg{m'}$ making the right-hand diagram above commute.
It follows that $\br{m',n'}\in\B$; just observe that $\Path{s'}(\co{m'})=\co{n'}$ because the composite $s'\circ \co{m'}$ is an embedding, and similarly $\Path{t'}(\co{n'})=\co{m'}$. 

It remains to show that $n\cvr n'$. 
For any element $x$ of a tree, denote by $\htf(x)$ its height. As $n\leq n'$ in $\Path{Y}$, it is enough to show that $\htf(n')=\htf(n)+1$. Using the fact that forest morphisms preserve the height of points and $\Path{s'}(\co{m'})=\co{n'}$, we get 
\[
\htf(n') = \htf(\co{n'}) =\htf(\co{m'})= \htf(m') = \htf(m)+1.
\]

Item~\ref{back} is proved in a similar way using the fact that $a$ is $k$-extendable.
\end{proof}

\begin{definition}
Consider a resource-indexed arboreal adjunction between $\E$ and $\C$, and an object $a\in\E$. For all $k > 0$, a \emph{$k$-extendable cover} of $a$ is a section $a \to a^*$ in $\E$ such that $a^*$ is $k$-extendable.
\end{definition}

\begin{proposition}\label{p:axiomatic-HP}
\ref{HP-abstract} holds for any resource-indexed arboreal adjunction between $\E$ and $\C$ satisfying the following properties for all $k >0$:
\begin{enumerate}[label=(\roman*)]
\item $\Cp^k$ is locally finite. 
\item $\E$ has binary products and each of its objects admits a $k$-extendable cover. 
\end{enumerate}
\end{proposition}
\begin{proof}
Fix a full subcategory $\D$ of $\E$ saturated under $\eqbCk$. For the non-trivial implication in \ref{HP-abstract}, assume $\D$ is closed under morphisms and let $a,b\in\E$ satisfy $a\prCk b$ and $a\in \D$. 

If $G_k\coloneqq L_k R_k$, then $a\prCk b$ implies $G_k a\to b$. Let $s\colon a\to a^*$ and $t\colon G_k a\to (G_k a)^*$ be sections with $a^*$ and $(G_k a)^*$ $k$-extendable objects. It follows from Lemma~\ref{l:equiv-rel-properties}\ref{hom-k-hom} that $a\eqaCk a^*$ and $G_k a\eqaCk (G_k a)^*$. By Lemma~\ref{l:equiv-rel-properties}\ref{k-equiv} we have $a\eqaCk G_k a$ and so, by transitivity, $a^*\eqaCk (G_k a)^*$. An application of Proposition~\ref{p:extendable} yields $a^*\eqbCk (G_k a)^*$. We thus have the following diagram, where $t^{-1}$ denotes the left inverse of $t$.
\[\begin{tikzcd}[column sep=0.1em]
a^* & {\eqbCk} & (G_k a)^* &&&&& &&&&& &&&&&  \\
a \arrow{u}{s} & {\eqaCk} & G_k a \arrow[leftarrow]{u}[swap]{t^{-1}} \arrow{rrrrrrrrrrrrrrr}{} &&&&& &&&&& &&&&&  b
\end{tikzcd}\] 
Since $a\in\D$, and $\D$ is saturated under $\eqbCk$ and closed under morphisms, all the objects in the diagram above sit in $\D$. In particular, $b\in \D$ and thus \ref{HP-abstract} holds.
\end{proof}

\subsection{Relativising to full subcategories}\label{s:relativisation}
Let $\E'$ be a full subcategory of $\E$.
We say that \ref{HP-abstract} holds \emph{relative to $\E'$} if the following condition is satisfied: For any full subcategory $\D$ of $\E$ saturated under $\eqbCk$, $\D \cap \E'$ is closed under morphisms in $\E'$ precisely when it is upwards closed in $\E'$ with respect to $\prCk$. Likewise for \ref{HPplus-abstract}.

For the next proposition, observe that in order for a comonad $G$ on $\E$ to restrict to a full subcategory $\E'$ of $\E$ it is necessary and sufficient that $Ga\in \E'$ for all $a\in \E'$.

\begin{proposition}\label{p:relative-HPT-and-HPTplus}
Consider a resource-indexed arboreal adjunction between $\E$ and $\C$, and let $\E'$ be a full subcategory of $\E$ such that the induced comonads $G_k\coloneqq L_k R_k$ restrict to $\E'$. If the resource-indexed arboreal adjunction has the bisimilar companion property then \ref{HP-abstract} holds relative to $\E'$. If it is idempotent then \ref{HPplus-abstract} holds relative to $\E'$.
\end{proposition}
\begin{proof}
The same, mutatis mutandis, as for Propositions~\ref{p:HPT-tame} and~\ref{p:HPT-graded}, respectively.
\end{proof}

Since the modal comonads $\Mk$ restrict to finite pointed Kripke structures, the previous result yields a variant of Theorem~\ref{th:hpt-graded-modal-logic} for finite structures:
\begin{theorem}\label{th:hpt-graded-modal-logic-finite}
The following statements are equivalent for any graded modal formula~$\phi$ of modal depth at most $k$ in a modal vocabulary:
\begin{enumerate}
\item $\phi$ is preserved under homomorphisms between finite pointed Kripke structures.
\item $\phi$ is logically equivalent over finite pointed Kripke structures to an existential positive modal formula of modal depth at most $k$.
\end{enumerate}
\end{theorem}

Similarly, since the guarded comonads $\mathbb{G}_{n,k}^{\g}$ restrict to finite structures, we obtain the following variant of Theorem~\ref{th:hpt-guarded} for finite structures:
\begin{theorem}\label{th:hpt-guarded-logics-finite}
Let $\g$ be a notion of guarding (either atom or loose).
The following statements are equivalent for any $\g$-guarded sentence~$\phi$ in $n$ variables of guarded-quantifier rank at most $k$ in a relational vocabulary:
\begin{enumerate}
\item $\phi$ is preserved under homomorphisms between finite structures.
\item $\phi$ is logically equivalent over finite structures to an existential positive $\g$-guarded sentence in $n$ variables of guarded-quantifier rank at most $k$.
\end{enumerate}
\end{theorem}

When the counits of the induced comonads $G_k\coloneqq L_k R_k$ are componentwise surjective, the next observation, combined with Proposition~\ref{p:relative-HPT-and-HPTplus}, gives a criterion to relativise equi-resource homomorphism preservation theorems to subclasses of structures. Recall that a \emph{negative formula} is one obtained from negated atomic formulas and $\vee, \wedge, \exists, \forall$.

\begin{lemma}\label{l:surj-counit-negpos-relat}
Let $G$ be a comonad on $\CS$ whose counit is componentwise surjective and let $T$ be a set of negative sentences. Then $G$ restricts to $\Mod(T)$.
\end{lemma}
\begin{proof}
If $\psi\in T$, its negation is logically equivalent to a positive sentence $\chi$. Positive sentences are preserved under surjective homomorphisms and so, for all $\As\in \CS$, considering the component of the counit $G\As\epi \As$ we obtain 
\[
G\As\models \chi \ \Longrightarrow \ \As \models \chi.
\]
\Iec $G$ restricts to $\Mod(\psi)$. As $\Mod(T)= \bigcap_{\psi\in T}{\Mod(\psi)}$, the statement follows.
\end{proof}

The counits of the guarded comonads $\mathbb{G}_{n,k}^{\g}$ are componentwise surjective, thus the equirank-variable homomorphism preservation theorem for guarded logics and its finite variant (Theorems~\ref{th:hpt-guarded} and~\ref{th:hpt-guarded-logics-finite}, respectively) admit a relativisation to any full subcategory of the form $\Mod(T)$ where $T$ is a set of negative $\g$-guarded sentences.

Relativisations to subclasses of structures can be obtained even in the absence of the bisimilar companion property; in that case, we need to ensure that $k$-extendable covers can be constructed within the subclass. We defer the statement of this result to Section~\ref{s:axioms-adj} (see Corollary~\ref{c:forcing-bisim-comp-prop-relative}). 

\section{An axiomatic approach}\label{s:axiomatic}

In this section, we consider resource-indexed arboreal adjunctions between $\E$ and~$\C$ that need not satisfy the bisimilar companion property. Using the notion of $k$-extendable cover introduced in Section~\ref{ss:extendability}, we identify sufficient conditions ensuring that property~\ref{HP-abstract} is satisfied (see Corollary~\ref{cor:HPT-axiomatic}). We introduce first conditions \ref{lim-colim}--\ref{fact-syst} on $\E$ in Section~\ref{s:axioms-extensional} and then, in Section~\ref{s:axioms-adj}, conditions \ref{paths-finite}--\ref{path-restriction-prop} on the adjunctions. In Section~\ref{s:equirank-proof}, we derive a slight generalisation of the equirank homomorphism preservation theorem by showing that these conditions are satisfied by the Ehrenfeucht-\Fraisse~resource-indexed arboreal adjunction. 

\begin{remark}\label{rem:does-not-apply-to}
Let us point out that we cannot derive from this axiomatic approach an \emph{equivariable} homomorphism preservation theorem, whereby the number of variables in a sentence is preserved, let alone an \emph{equirank-variable} one. In fact, the corresponding ($k$-round) $n$-pebble comonads do not satisfy property~\ref{path-emb} below. On the other hand, under the additional assumption that $k\leq n+2$, where $k$ is the quantifier rank and $n$ is the number of variables in a sentence, an equirank-variable homomorphism preservation theorem was proved by Paine in~\cite{Paine2020}. Also, our approach does not readily apply to \emph{hybrid logic} because the hybrid comonads $\mathbb{H}_k$ in~\cite{Hybrid2022} do not satisfy the path restriction property~\ref{path-restriction-prop} (\cf Definition~\ref{def:path-re}).
\end{remark}

\subsection{Axioms for the extensional category}\label{s:axioms-extensional}
We require that the category $\E$ have the following properties:
\begin{enumerate}[label=\textnormal{(E\arabic*)}]
\item\label{lim-colim} $\E$ has all finite limits and small colimits.
\item\label{fact-syst} $\E$ is equipped with a proper factorisation system such that:
\begin{itemize}[leftmargin=*]
\item Embeddings are stable under pushouts along embeddings.
\item Pushout squares of embeddings are also pullbacks.
\item Pushout squares of embeddings are stable under pullbacks along embeddings.
\end{itemize}
\end{enumerate}

\begin{remark}
Note that property~\ref{fact-syst} only involves one half of the factorisation system, namely the embeddings. In fact, it could be weakened to the requirement that $\E$ admit a class of monomorphisms $\mathscr{N}$ satisfying appropriate properties. When $\mathscr{N}$ is the class of all monomorphisms, these are akin to the conditions for an \emph{adhesive category}, \cf \cite{Adhesive2004}.
\end{remark}

\begin{example}\label{ex:structures-axioms}
If $\sg$ is a relational vocabulary, $\CS$ satisfies \ref{lim-colim}--\ref{fact-syst}. In fact, it is well known that $\CS$ is complete and cocomplete, hence it satisfies~\ref{lim-colim}. For~\ref{fact-syst}, consider the proper factorisation system given by surjective homomorphisms and embeddings. 
Up to isomorphism, embeddings can be identified with inclusions of induced substructures. The pushout of a span of embeddings in $\CS$ can be identified with a union of structures, and so embeddings are stable under pushouts along embeddings. The remaining two conditions in~\ref{fact-syst} hold because they are satisfied in $\Set$ by the class of monomorphisms, and the forgetful functor $\CS\to \Set$ preserves and reflects pullback and pushout diagrams consisting of embeddings.
\end{example}

We note in passing that properties \ref{lim-colim}--\ref{fact-syst} are stable under taking coslices:

\begin{lemma}\label{l:extensional-ax-coslices}
If a category $\E$ satisfies~\ref{lim-colim}--\ref{fact-syst}, then so does $e/{\E}$ for all $e\in \E$.
\end{lemma}
\begin{proof}
Fix an arbitrary object $e\in\E$. It is well known that limits and colimits in $e/{\E}$ are inherited from $\E$, so $e/{\E}$ satisfies~\ref{lim-colim} because $\E$ does.

By assumption, $\E$ admits a proper factorisation system satisfying~\ref{fact-syst}. Let $\Q$ and $\M$ be the classes of arrows in $e/{\E}$ whose underlying morphisms in $\E$ are quotients and embeddings, respectively. It is folklore that $(\Q,\M)$ is a weak factorisation system in $e/{\E}$. Moreover, this factorisation system is proper because the codomain functor $\cod\colon e/{\E} \to \E$ is faithful.
Recall that $\cod\colon e/{\E} \to \E$ preserves pushouts, so embeddings in $e/{\E}$ are stable under pushouts along embeddings because the corresponding property is satisfied in~$\E$. The remaining two properties in~\ref{fact-syst} follow by similar reasoning, using the fact that $\cod\colon e/{\E} \to \E$ preserves and reflects limits and pushouts.
\end{proof}

\begin{example}
It follows from Example~\ref{ex:structures-axioms} and Lemma~\ref{l:extensional-ax-coslices} that, for all relational vocabularies $\sg$, the category $\CSstar$ of pointed $\sg$-structures satisfies \ref{lim-colim}--\ref{fact-syst}.
\end{example}

\subsection{Axioms for the resource-indexed adjunctions}\label{s:axioms-adj}
We now assume that the extensional category $\E$ satisfies \ref{lim-colim}--\ref{fact-syst}, and proceed to introduce conditions on the resource-indexed arboreal adjunction between $\E$ and $\C$.
To start with, consider an arbitrary adjunction $L\dashv R \colon \E \to \C$. As with any adjunction, there are hom-set bijections
\begin{equation*}
\E(L c, e) \longrightarrow \C(c, R a), \enspace f\mapsto f^\flat
\end{equation*}
natural in $c\in\C$ and $e\in\E$. Explicitly, $f^\flat$ is defined as the composite 
\[\begin{tikzcd}
c \arrow{r}{\eta_c} & R L c \arrow{r}{R f} & R e
\end{tikzcd}\]
where $\eta$ is the unit of the adjunction $L\,{\dashv}\, R$. The inverse of the function $f\mapsto f^\flat$ sends $g\in \C(c, R e)$ to the morphism $g^\#$ given by the composition
\[\begin{tikzcd}
L c \arrow{r}{L g} & L R e \arrow{r}{\epsilon_{e}} & e,
\end{tikzcd}\]
where $\epsilon$ is the counit of the adjunction. 
Naturality of these bijections means that
\[
(f_1\circ f_2)^\flat = R f_1 \circ f_2^\flat
\]
for all morphisms $f_1\colon e\to e'$ and $f_2\colon L c \to e$ in $\E$, and
\[
(g_1\circ g_2)^\#=g_1^\# \circ L g_2
\]
for all morphisms $g_1\colon c\to R e$ and $g_2\colon c'\to c$ in $\C$.

Next, we introduce the path restriction property for resource-indexed arboreal adjunctions. In a nutshell, this states that whenever $a\in \E$ embeds into the image under $L_k$ of a path, $a$ itself can be equipped with a path structure. Furthermore, these path structures can be chosen in a coherent fashion. We start with an auxiliary definition:
\begin{definition}
Consider a resource-indexed arboreal adjunction between $\E$ and $\C$. A path $P\in\Cp^k$ is \emph{smooth} if there exist $e\in \E$ and an embedding $P\emb R_k e$.
\end{definition}

\begin{remark}
The motivation for considering smooth paths arises from the fact that, when considering a fresh binary relation symbol $I$ modelling equality in the logic (\cf Example~\ref{ex:res-ind-arb-adj}), the interpretation of $I$ in these paths is always an equivalence relation.
\end{remark}

\begin{definition}\label{def:path-re}
A resource-indexed arboreal adjunction between $\E$ and $\C$ has the \emph{path restriction property} if, for all smooth paths $Q\in\Cp^k$ and embeddings ${j\colon a\emb L_k Q}$, there is a path $Q_a\in\Cp^k$ such that $L_k Q_a \cong a$ and the following conditions are satisfied:
\begin{enumerate}[label=(\roman*)]
\item\label{path-re-1} For all path embeddings $!_{P,Q}\colon P\emb Q$ in $\C_k$ and all commutative diagrams
\[\begin{tikzcd}
L_k P \arrow[rr, relay arrow=2ex, "L_k(!_{P,Q})"] \arrow[rightarrowtail]{r}{f}  & a \arrow[rightarrowtail]{r}{j} & L_k Q
\end{tikzcd}\]
there is an arrow $\ell\colon P\to Q_{a}$ such that $L_k \ell = f$.
\item\label{path-re-2} For all path embeddings $!_{P,Q}\colon P\emb Q$ such that $L_k(!_{P,Q})$ is an embedding, the pullback of $L_k(!_{P,Q})$ along $j$ is of the form $L_k \ell$ for some $\ell\colon Q_b\to Q_a$. 
\[\begin{tikzcd}
b \arrow[dr, phantom, "\lrcorner", very near start] \arrow[rightarrowtail]{r} \arrow[rightarrowtail]{d}[swap]{L_k \ell} & L_k P \arrow[rightarrowtail]{d}{L_k (!_{P,Q})} \\
a \arrow[rightarrowtail]{r}{j} & L_k Q
\end{tikzcd}\]
\end{enumerate}
\end{definition}

Finally, recall that an object $a$ of a category $\A$ is \emph{finitely presentable} if the associated hom-functor $\A(a,-)\colon \A\to \Set$ preserves directed colimits.

With regards to the resource-indexed arboreal adjunction, we assume that the following properties are satisfied for all $k > 0$ and all paths $P\in\Cp^k$:
\begin{enumerate}[label=\textnormal{(A\arabic*)}]
\item\label{paths-finite} The category $\Cp^k$ is locally finite and has finitely many objects up to isomorphism.
\item\label{paths-Lk-fp} $L_k P$ is finitely presentable in $\E$.
\item\label{path-emb} For all arrows $m\colon P\to R_k a$ in $\C_k$, if $m$ is an embedding then so is $m^\# \colon L_k P\to a$. The converse holds whenever $P$ is smooth.
\item\label{path-restriction-prop} The path restriction property is satisfied.
\end{enumerate}

\begin{theorem}\label{t:model-construction}
Consider a resource-indexed arboreal adjunction between $\E$ and $\C$ satisfying \ref{lim-colim}--\ref{fact-syst} and \ref{paths-finite}--\ref{path-restriction-prop}. For all $a\in\E$ and all $k > 0$, there exists a $k$-extendable cover of $a$.
\end{theorem}

The proof of the previous key fact is deferred to Section~\ref{s:proof-mc}. Let us point out the following immediate consequence:
\begin{corollary}\label{cor:HPT-axiomatic}
\ref{HP-abstract} holds for all resource-indexed arboreal adjunctions satisfying \ref{lim-colim}--\ref{fact-syst} and \ref{paths-finite}--\ref{path-restriction-prop}.
\end{corollary}
\begin{proof}
By Proposition~\ref{p:axiomatic-HP} and Theorem~\ref{t:model-construction}.
\end{proof}

We can also deduce the following relativisation result. 
Let us say that a full subcategory $\D$ of a category $\A$ is \emph{closed (in $\A$) under co-retracts} provided that, whenever $A\in \D$ and there is a section $A\to B$ in $\A$, also $B\in \D$.
\begin{corollary}\label{c:forcing-bisim-comp-prop-relative}
Let $\sg$ be a relational vocabulary and consider a resource-indexed arboreal adjunction between $\CS$ and $\C$ satisfying \ref{paths-finite}--\ref{path-restriction-prop}. The following hold:
\begin{enumerate}
\item\label{rel-co-retr} If $\D$ is a full subcategory of $\CS$ closed under co-retracts such that each induced comonad $G_k\coloneqq L_k R_k$ restricts to $\D$, then \ref{HP-abstract} holds relative to $\D$.
\item\label{rel-negated-pos} If the counits of the comonads $G_k$ are componentwise surjective and $T$ is a set of negative sentences in the vocabulary $\sg$, then \ref{HP-abstract} holds relative to $\Mod(T)$.
\end{enumerate}
\end{corollary}

\begin{proof}
The proof of item~\ref{rel-co-retr} is the same, mutatis mutandis, as for Proposition~\ref{p:axiomatic-HP}, using Theorem~\ref{t:model-construction} and the fact that if $\D$ is closed under co-retracts then $k$-extendable covers can be constructed within $\D$. 

Item~\ref{rel-negated-pos} is an immediate consequence of item~\ref{rel-co-retr}. 
Just observe that the comonads $G_k$ restrict to $\Mod(T)$ by Lemma~\ref{l:surj-counit-negpos-relat}, and $\Mod(T)$ is closed under co-retracts (\cf the proof of the aforementioned lemma).
\end{proof}

\subsection{The equirank homomorphism preservation theorem}\label{s:equirank-proof}
As observed in Section~\ref{s:HP-HPplus}, Rossman's equirank homomorphism preservation theorem is equivalent to property~\ref{HP-abstract} for the Ehrenfeucht-\Fraisse~resource-indexed arboreal adjunction between $\CS$ and $\RT(\sg^I)$. In turn, by Corollary~\ref{cor:HPT-axiomatic}, to establish~\ref{HP-abstract} it suffices to show that the latter resource-indexed arboreal adjunction satisfies \ref{lim-colim}--\ref{fact-syst} and \ref{paths-finite}--\ref{path-restriction-prop}. By Example~\ref{ex:structures-axioms}, the category $\CS$ satisfies \ref{lim-colim}--\ref{fact-syst} when equipped with the (surjective homomorphisms, embeddings) factorisation system, so it remains to show that \ref{paths-finite}--\ref{path-restriction-prop} hold. Before doing so, note that Corollary~\ref{c:forcing-bisim-comp-prop-relative} yields the following slight generalisation of the equirank homomorphism preservation theorem (just observe that the counits of the induced comonads on $\CS$ are componentwise surjective).
\begin{theorem}\label{t:equirank-hpt-relative}
Let $\sg$ be a relational vocabulary and let $\D$ be a full subcategory of $\CS$ closed under co-retracts such that the comonads on $\CS$ induced by the Ehrenfeucht-\Fraisse~resource-indexed arboreal adjunction restrict to $\D$. Then the equirank homomorphism preservation theorem holds relative to $\D$.

In particular, the equirank homomorphism preservation theorem holds relative to $\Mod(T)$ whenever $T$ is a set of negative sentences in the vocabulary $\sg$.
\end{theorem}

\begin{remark}
A consequence of the first part of Theorem~\ref{t:equirank-hpt-relative} is that the equirank homomorphism preservation theorem admits a relativisation to any class of structures that is \emph{co-homomorphism closed}, \ie downwards closed with respect to the homomorphism preorder on $\CS$, a fact already pointed out by Rossman in~\cite[\S 7.1.2]{Rossman2008}.
\end{remark}

We proceed to verify conditions \ref{paths-finite}--\ref{path-restriction-prop} for the adjunctions $L_k\dashv R_k$, where 
\[
L_k \coloneqq H \LE_k \ \text{ and } \ R_k\coloneqq \RE_k J
\]
with the notation of Example~\ref{ex:res-ind-arb-adj}.
Note that, since a first-order sentence contains only finitely many relation symbols, in order to deduce the equirank homomorphism preservation theorem, as well as Theorem~\ref{t:equirank-hpt-relative} above, we can assume without loss of generality that the relational vocabulary $\sg$ is finite.

\ref{paths-finite} Recall from Example~\ref{ex:RE} that, for all $k>0$, the paths in $\RTk(\sg^I)$ are those forest-ordered $\sg^I$-structures $(\As,\leq)$ such that the order is a chain of cardinality at most~$k$. Thus, $\As$ has cardinality at most $k$. It follows at once that there are finitely many paths in $\RTk(\sg^I)$ up to isomorphism, and at most one arrow between any two paths.

\ref{paths-Lk-fp} For any path $P=(\As,\leq)$ in $\RTk(\sg^I)$, $L_k P$ is the quotient of the $\sg$-reduct of $\As$ with respect to the equivalence relation generated by $I^{\As}$. As $\As$ is finite, so is $L_k P$. The finitely presentable objects in $\CS$ are precisely the finite $\sg$-structures (see \eg \cite[\S 5.1]{AR1994}), hence $L_k P$ is finitely presentable.

\ref{path-emb} Consider an arrow $m\colon P\to R_k \Bs$ in $\RTk(\sg^I)$, with $P=(\As,\leq)$ a path and $\Bs$ a $\sg$-structure. 
Let $\Bs'\coloneqq J(\Bs)$ be the $\sg^I$-structure obtained from $\Bs$ by interpreting $I$ as the identity relation.
Then $R_k \Bs$ is obtained by equipping $\Ek(\Bs')$ with the prefix order. Consider the $\sg$-homomorphism
\[
L_k m = H m \colon H(\As)\to H\Ek(\Bs').
\]
For convenience of notation, given an element $x\in\As$ we write $[x]$ for the corresponding element of $H(\As)$, and likewise for elements of $\Ek(\Bs')$.
Then $m^\#\colon L_k P \to \Bs$ is the composite of $Hm$ with the homomorphism $H\Ek(\Bs') \to \Bs$ sending the equivalence class of an element of $\Ek(\Bs')$ to the last element of any of its representatives. This map is well defined because, for any pair of sequences in $I^{\Ek(\Bs')}$, their last elements coincide.

Suppose $m$ is an embedding. If $m^\#([x])=m^\#([y])$ then $(m(x),m(y))$ belongs to the equivalence relation generated by $I^{\Ek(\Bs')}$, and so $(x,y)$ belongs to the equivalence relation generated by $I^{\As}$. It follows that $[x]=[y]$ in $H(\As)$, and so $m^\#$ is injective. The same argument, mutatis mutandis, shows that $m^\#$ reflects the interpretation of the relation symbols, hence is an embedding. 

Conversely, suppose that $m^\#$ is an embedding and $P$ is smooth. Consider an embedding $n\colon P\emb R_k \Cs$ with $\Cs\in \CS$. Note that the restriction of $I^{\Ek(\Cs')}$ to the image of $n$ is an equivalence relation, hence $I^{\As}$ is an equivalence relation. As any forest morphism whose domain is linearly ordered is injective, $m$ is injective. So, it remains to show that $m$ reflects the interpretation of the relation symbols. For all $x,y\in \As$, if 
\[
(m(x),m(y))\in I^{\Ek(\Bs')}
\] then $m^\#([x])=m^\#([y])$ and so $[x]=[y]$ because $m^\#$ is injective. That is, $(x,y)\in I^{\As}$, showing that $m$ reflects the interpretation of the relation $I$. Suppose now that $S$ is a relation symbol different from~$I$. For convenience of notation we shall assume that $S$ has arity $2$; the general case is a straightforward adaptation. For all $x,y\in \As$, if 
\[
(m(x),m(y))\in S^{\Ek(\Bs')}
\] 
then $(m^\#([x]), m^\#([y]))\in S^{\Bs}$ and so $([x], [y])\in S^{H(\As)}$ because $m^\#$ is an embedding. That is, there are $x',y'\in \As$ such that $(x,x'),(y,y')\in I^{\As}$ and $(x',y')\in S^{\As}$. We claim that the following property holds, from which it follows that $m$ is an embedding:
\begin{equation}\label{eq:smooth-property}
(x,x'),(y,y')\in I^{\As} \ \text{ and } \ (x',y')\in S^{\As} \ \Longrightarrow \ (x,y)\in S^{\As}. \tag{$\ast$}
\end{equation}
In turn, this is a consequence of the fact that $n$ is an embedding and, in $\Ek(\Cs')$, 
\[
(n(x),n(x')),(n(y),n(y'))\in I^{\Ek(\Cs')} \ \text{ and } \ (n(x'),n(y'))\in S^{\Ek(\Cs')}
\]
imply $(n(x),n(y))\in S^{\Ek(\Cs')}$.

\ref{path-restriction-prop} Finally, we show that the path restriction property is satisfied. Let $Q=(\Bs,\leq)$ be a smooth path in $\RTk(\sg^I)$, and let $j\colon \As\emb H(\Bs)$ be an embedding in $\CS$. Without loss of generality, we can identify $\As$ with a substructure of $H(\Bs)$, and $j$ with the inclusion map. As observed above, since $Q$ is smooth, $I^{\Bs}$ is an equivalence relation and property~\eqref{eq:smooth-property} is satisfied (with $\Bs$ in place of $\As$). If $q_{\Bs}\colon \Bs\twoheadrightarrow H(\Bs)$ is the canonical quotient map, let $Q_{\As}$ denote the substructure of $\Bs$ whose underlying set is
\[
\{x\in \Bs \mid q_{\Bs}(x) \in \As\}.
\]
Then $Q_{\As}$ is a path in $\RTk(\sg^I)$ when equipped with the restriction of the order on~$\Bs$ and, using~\eqref{eq:smooth-property}, we get $H(Q_{\As})\cong\As$. 

It follows from the definition of $Q_{\As}$ that item~\ref{path-re-1} in Definition~\ref{def:path-re} is satisfied. Just observe that any substructure of $Q_{\As}$ that is downwards closed in $Q$ is also downwards closed in $Q_{\As}$. With regards to item~\ref{path-re-2}, consider a path embedding $!_{P,Q}\colon P\emb Q$ with $P=(\Cs,\leq)$ and form the following pullback square in $\CS$.
\[\begin{tikzcd}
\Ds \arrow[dr, phantom, "\lrcorner", very near start] \arrow[rightarrowtail]{r} \arrow[rightarrowtail]{d} & H(\Cs) \arrow[rightarrowtail]{d}{L_k (!_{P,Q})} \\
\As \arrow[rightarrowtail]{r}{j} & H(\Bs)
\end{tikzcd}\]
Identifying $H(\Cs)$ with a substructure of $H(\Bs)$, we can assume $\Ds$ is the intersection of $\As$ and $H(\Cs)$. 
Because $\Cs$ is a substructure of $\Bs$, it follows that $Q_{\Ds}$ is a substructure of~$Q_{\As}$. Moreover, because $\Cs$ is downwards closed in~$\Bs$, $Q_{\Ds}$ is downwards closed in~$Q_{\As}$. That is, there is an inclusion $Q_{\Ds}\emb Q_{\As}$ whose image under $L_k$ coincides with the pullback of $L_k (!_{P,Q})$ along $j$. Hence the path restriction property holds.

\section{Proof of Theorem~\ref{t:model-construction}}\label{s:proof-mc}
%
For the remainder of this section, we fix an arbitrary resource-indexed arboreal adjunction between $\E$ and~$\C$, with adjunctions $L_k\dashv R_k\colon \E\to \C_k$, satisfying \ref{lim-colim}--\ref{fact-syst} and \ref{paths-finite}--\ref{path-restriction-prop}.

\subsection{Relative extendability}\label{ss:relative-ext}

For all $k > 0$, we denote by $\o{\C}_k$ the full subcategory of $\C_k$ whose objects are colimits of finite diagrams of embeddings in $\Cp^k$. Further, we write $L_k[\o{\C}_k]$ for the full subcategory of $\E$ defined by the objects of the form $L_k c$ for $c\in \o{\C}_k$. 

\begin{remark}\label{rem:finitely-many-finite-colim}
It follows from~\ref{paths-finite} that $\o{\C}_k$ is equivalent to a finite category. Therefore $L_k[\o{\C}_k]$ contains, up to isomorphism, only finitely many objects.
\end{remark}

As we shall see in the following lemma, every path embedding $P\emb R_k a$ is homomorphically equivalent to one of the form $P\emb R_k \tilde{a}$ with $\tilde{a}\in L_k[\o{\C}_k]$. Consequently, in the definition of $k$-extendable object (see Definition~\ref{d:k-extendable}) we can assume without loss of generality that $e\in L_k[\o{\C}_k]$. This observation, combined with Remark~\ref{rem:finitely-many-finite-colim}, will allow us to control the size of the diagrams featuring in the proof of Theorem~\ref{t:model-construction}.
\begin{lemma}\label{l:fg-equiv-finite}
For all path embeddings $m\colon P\emb R_k a$, there are $\tilde{a}\in L_k[\o{\C}_k]$ and a path embedding $\tilde{m}\colon P\emb R_k \tilde{a}$ such that $m\rl \tilde{m}$ in $P/{\C_k}$.
\end{lemma}
\begin{proof}
Fix an arbitrary path embedding $m\colon P\emb R_k a$. By \ref{paths-finite}, there is a finite set of paths $\mathscr{P}=\{P_1,\ldots,P_j\}\subseteq \Cp^k$ such that each path in $\C_k$ is isomorphic to exactly one member of $\mathscr{P}$. We can assume without loss of generality that $P\in\mathscr{P}$.

For each path embedding $p\in \Path{(R_k a)}$, denote by $T_p$ the tree obtained by first considering the tree ${\uparrow} p\subseteq \Path{(R_k a)}$ and then replacing each node $q$ (which is an isomorphism class of a path embedding) with the unique path $P_i\in\mathscr{P}$ such that $P_i\cong \dom(q)$. We assume that $T_p$ is \emph{reduced}, \ie given any two nodes $x$ and $y$ of $T_p$ that cover the same node, if the trees ${\uparrow} x$ and ${\uparrow} y$ are equal then $x=y$. (If $T_p$ is not reduced, we can remove branches in the obvious manner to obtain a maximal reduced subtree $T'_p$.)
We refer to $T_p$ as the \emph{type} of $p$; note that this is a finite tree. 
In particular, if $\bot$ is the root of $\Path{(R_k a)}$, we get a finite tree $T_{\bot}$.

Now, for each node $x$ of $T_\bot$, we shall define a path embedding $m_x$ into $R_k a$ whose domain belongs to $\mathscr{P}$. The definition of $m_x$ is by induction on the height of $x$.
Suppose $x$ has height $0$, \ie $x$ is the root of $T_\bot$. Then $x=P_i$ for a unique $i\in \{1,\ldots,j\}$. Define $m_x$ as the restriction of $m$ to $P_i$, \ie the composition of $m\colon P\emb R_k a$ with the unique embedding $P_i\emb P$. Next, suppose $m_z$ has been defined for all nodes $z$ of height at most $l$, and let $x$ be a node of height $l+1$ labeled by some $P_j$. We distinguish two cases:
\begin{itemize}
\item If there is a node $y\geq x$ such that $T_m$ coincides with the tree ${\uparrow} y\subseteq T_{\bot}$, then we let $m_x$ be the restriction of $m$ to $P_j$. 
Note that, in this case, the type of $m_x$ coincides with the tree ${\uparrow} x\subseteq T_{\bot}$. Moreover, if $z$ is the predecessor of $x$ then $m_z$ will also be an appropriate restriction of $m$, and thus $m_x$ extends $m_z$.
\item Otherwise, we let $m_x\colon P_j\emb R_k a$ be any path embedding such that:
\begin{enumerate}[label=(\roman*)]
\item The type of $m_x$ coincides with the tree ${\uparrow} x\subseteq T_{\bot}$.
\item $m_x$ extends $m_z$, where $z$ is the predecessor of $x$.
\end{enumerate} 
Note that such an embedding $m_x$ exists because $x\in {\uparrow} z\subseteq T_{\bot}$ and, by inductive hypothesis, ${\uparrow} z$ coincides with the type of $m_z$.
\end{itemize}

The set 
\[
V\coloneqq \{m_x\mid x\in T_\bot\}
\]
is finite and contains $m$. We regard $V$ as a cocone over a finite diagram $D$ of paths and embeddings between them. Let $\tilde{a}\coloneqq L_k(\colim D)$ and note that $\tilde{a}\in L_k[\o{\C}_k]$. The functor $L_k$ preserves colimits because it is left adjoint, hence $\tilde{a}$ is the colimit in $\E$ of the diagram $L_k D$.
The cocone $\{n^\#\mid n\in V\}$ with vertex $a$ over $L_k D$ then factors through a unique morphism $f\colon \tilde{a}\to a$. 
By construction, $m\colon P\emb R_k a$ factors through $R_k f$, and so there is $\tilde{m}\colon P\emb \tilde{a}$ such that $m=R_k f\circ \tilde{m}$. Hence, $\tilde{m}\to m$.

Next, with the aim of showing that $m\to \tilde{m}$, we shall define a morphism ${R_k a \to R_k \tilde{a}}$. As $R_k a$ is path generated, it suffices to define a cocone 
\[
W=\{\phi_p\mid p\in \Path{(R_k a)}\}
\] 
with vertex $R_k \tilde{a}$ over the diagram of path embeddings into $R_k a$. Suppose $p\in \Path{(R_k a)}$. We define the corresponding arrow $\phi_p$  by induction on the height of $p$:
\begin{enumerate}[label=(\roman*)]
\item If $p$ is the root of $\Path{(R_k a)}$, then it factors through $R_k f\colon R_k \tilde{a}\emb R_k a$, and so it yields an embedding $\phi_p\colon \dom(p)\emb R_k \tilde{a}$.
\item Suppose that $p$ has height $l+1$ and $\phi_q$ has been defined whenever $q$ has height at most $l$. We distinguish two cases: if $p$ factors through $R_k f\colon R_k \tilde{a}\emb R_k a$, \ie $p= R_k f\circ s_p$ for some embedding $s_p$, then we set $\phi_p\coloneqq s_p$. This is the case, in particular, when $p\leq m$ in $\Path{(R_k a)}$. Clearly, if $p$ extends $q$ then $\phi_p$ extends~$\phi_q$.

Otherwise, let $q$ be such that $p\succ q$. By inductive hypothesis, we can suppose that $R_k f\circ \phi_q$ coincides with an embedding $m_x\colon P_j\emb R_k a$ in $V$ (up to an isomorphism $\dom(q)\cong P_j$) whose type coincides with the tree ${\uparrow} x\subseteq T_{\bot}$. As $q$ corresponds to a node $y$ covering $x$ labeled by some $P_h\cong \dom(p)$, by definition of $V$ there is an embedding $m_y\colon P_h\emb R_k a$ in $V$ such that $m_y$ extends $m_x$, and the type $m_y$ coincides with ${\uparrow} y$. Since $m_y$ factors through $R_k f$, precomposing with the isomorphism $\dom(p)\cong P_h$ we get an embedding $\phi_p\colon \dom(p)\emb R_k \tilde{a}$. Observe that $\phi_p$ extends $\phi_q$.
\end{enumerate} 
The compatibility condition for the cocone $W$ states that $\phi_p$ extends $\phi_q$ whenever $p$ extends $q$, which is ensured by the definition above.
Thus, $W$ induces a morphism $g\colon R_k a \to R_k \tilde{a}$ and, by construction, $g\circ m= \tilde{m}$. Hence, $m\to \tilde{m}$.
\end{proof}

The construction of $k$-extendable covers is akin to that of $\omega$-saturated elementary extensions in model theory, where one starts with a first-order structure $M$ and constructs an elementary extension $M_1$ of $M$ in which all types over (finite subsets of) $M$ are realised, then an elementary extension $M_2$ of $M_1$ in which all types over $M_1$ are realised, and so forth. The union of the induced elementary chain of models yields the desired $\omega$-saturated elementary extension of $M$.

In the same spirit, we introduce a notion of $k$-extendability relative to a homomorphism, which models the one-step construction just outlined.

\begin{definition}\label{def:relative-extendability}
Let $h\colon a\to b$ be an arrow in $\E$. We say that $b$ is \emph{$k$-extendable relative to $h$} if the following property is satisfied for all $e\in L_k[\o{\C}_k]$: For all path embeddings $m\colon P\emb R_k a$ and $n\colon P\emb R_k e$ such that $\co{m}\rl \co{n}$,
\[\begin{tikzcd}[column sep=2em]
{} & P \arrow[rightarrowtail]{dl}[swap]{\co{m}} \arrow[rightarrowtail]{dr}{\co{n}} & {} \\
\Sg{m} \arrow[yshift=3pt]{rr} & & \Sg{n} \arrow[yshift=-3pt]{ll}
\end{tikzcd}\]
if $n'\colon Q\emb R_k e$ is a path embedding such that $n\leq n'$ in $\Path{(R_k e)}$, there is a path embedding $m'\colon Q\emb R_k b$ such that the leftmost diagram below commutes and $\co{m'}\rl \co{n'}$.
\begin{center}
\begin{tikzcd}
P \arrow[rightarrowtail]{r}{!} \arrow[rightarrowtail]{d}[swap]{m} & Q \arrow[rightarrowtail,dashed]{d}{m'} \\
R_k a \arrow{r}{R_k h} & R_k b
\end{tikzcd}
\ \ \ \ \ \ \ 
\begin{tikzcd}[column sep=2em]
{} & Q \arrow[rightarrowtail]{dl}[swap]{\co{m'}} \arrow[rightarrowtail]{dr}{\co{n'}} & {} \\
\Sg{m'} \arrow[yshift=3pt,dashed]{rr} & & \Sg{n'} \arrow[yshift=-3pt,dashed]{ll}
\end{tikzcd}
\end{center}
\end{definition}

Suppose that, given an object $a\in\E$, we are able to construct a section $s\colon a\to b$ such that $b$ is $k$-extendable relative to $s$. Iterating this process countably many times, we obtain a $k$-extendable cover $a\to a^*$, thus settling Theorem~\ref{t:model-construction}. The main hurdle consists in establishing the following proposition; a proof is offered in Section~\ref{s:proof-one-step-ext}.
\begin{proposition}\label{pr:relative-extension}
For all $a\in\E$ and all $k > 0$ there is a section $s\colon a\to b$ such that $b$ is $k$-extendable relative to $s$.
\end{proposition}

We can finally prove Theorem~\ref{t:model-construction}:
\begin{proof}[Proof of Theorem~\ref{t:model-construction}]
Let $a\in \E$. By applying Proposition~\ref{pr:relative-extension} repeatedly, we obtain a chain of sections
\[\begin{tikzcd}
a \arrow{r}{s_1} & b_1 \arrow{r}{s_2} & b_2 \arrow{r}{s_3} & b_3 \arrow{r}{s_4} & \cdots 
\end{tikzcd}\]
such that $b_i$ is $k$-extendable relative to $s_i$, for all $i\geq 1$. Denote the previous diagram by $D$ and let $a^*$ be the colimit of $D$ in $\E$, which exists by~\ref{lim-colim}. Let $h_i\colon b_i\to a^*$ be the colimit map with domain $b_i$, and $s\colon a\to a^*$ the one with domain~$a$. As all the arrows in $D$ are sections, so are the colimit maps; in particular, $s$ is a section. 

We claim that $a^*$ is $k$-extendable. Suppose $m\colon P\emb R_k(a^*)$ and $n\colon P\emb R_k e$ are path embeddings such that $\co{m}\rl \co{n}$. 
By Lemma~\ref{l:fg-equiv-finite}, we can assume without loss of generality that $e\in L_k[\o{\C}_k]$.

By~\ref{paths-Lk-fp}, $L_k P$ is finitely presentable in $\E$ and so $m^\#\colon L_k P \to a^*$ factors through one of the colimit maps. Assume without loss of generality that $m^\#$ factors through $h_j \colon b_j \to a$ for some $j\,{\geq}\, 1$, so there is an arrow $r\colon L_k P\to b_j$ satisfying $m^\# = h_j \circ r$. If $m_j\coloneqq r^\flat$, it follows that $m = R_k h_j \circ m_j$. In particular, $m_j$ is an embedding because so is $m$. Since $R_k h_j$ is a section, Remark~\ref{rem:sections-coslice} entails $m_j\rl m$, and so $\co{(m_j)}\rl \co{m}$ by Lemma~\ref{l:corestriction-properties}\ref{corestriction-arrows}. Because $\co{m}\rl\co{n}$, also $\co{(m_j)}\rl \co{n}$.

Now, let $n'\colon Q\emb R_k e$ be any path embedding such that $n\leq n'$ in~$\Path{(R_k e)}$. Since $b_{j+1}$ is $k$-extendable relative to $s_{j+1}$, there is a path embedding $m'\colon Q\emb R_k b_{j+1} $ such that $\co{m'}\rl \co{n'}$ and the following diagram commutes.
\begin{equation*}
\begin{tikzcd}[column sep=4em]
P \arrow[rightarrowtail]{r}{!} \arrow[rightarrowtail]{d}[swap]{m_j} & Q \arrow[rightarrowtail]{d}{m'} \\
R_k b_j \arrow{r}{R_k s_{j+1}} & R_k b_{j+1}
\end{tikzcd}
\end{equation*}
It follows that 
\[
m = R_k h_j \circ m_j = R_k h_{j+1} \circ R_k s_{j+1} \circ m_j = R_k h_{j+1} \circ m'\circ {!}
\]
and so $m''\coloneqq R_k h_{j+1} \circ m'\colon Q\emb R_k(a^*)$ satisfies $m\leq m''$ in $\Path{(R_k(a^*))}$. Again by Remark~\ref{rem:sections-coslice} and Lemma~\ref{l:corestriction-properties}\ref{corestriction-arrows} we get $\co{m''}\rl \co{m'}$, and thus $\co{m''}\rl \co{n'}$. This shows that $a^*$ is $k$-extendable. 
\end{proof}

\subsection{Proof of Proposition~\ref{pr:relative-extension}}\label{s:proof-one-step-ext}
%
Fix an object $a$ of $\E$ and a positive integer $k$.
To improve readability, we drop the subscript from $L_k$ and $R_k$, and simply write $L$ and $R$ (but continue to denote by $\C_k$ the arboreal category).
We must find a section $s\colon a\to b$ such that $b$ is $k$-extendable relative to $s$. Consider all pairs of path embeddings
\[
(u\colon P\emb Ra, v\colon P\emb Re)
\] 
in $\C_k$ such that $e\in L[\o{\C}_k]$ and $L\co{v}\to u^\#$ in $LP/{\E}$. 
\begin{remark}
Note that $L\co{v}\to u^\#$ entails that $L\co{v}$ is an embedding. Just observe that $u^\#$ is an embedding by the first part of~\ref{path-emb}.
\end{remark}
Each such pair $(u,v)$ induces a pushout square in $\E$ as follows.
\[\begin{tikzcd}[column sep=3em]
L P \arrow[rightarrowtail]{d}[swap]{u^\#} \arrow[rightarrowtail]{r}{L\co{v}} & L\Sg{v} \arrow[rightarrowtail]{d}{\lambda_{(u,v)}} \\
a \arrow[rightarrowtail]{r}{\iota_{(u,v)}} & a +_{L P} L \Sg{v} \arrow[ul, phantom, "\ulcorner", very near start]
\end{tikzcd}\]
This pushout square exists by~\ref{lim-colim} and consists entirely of embeddings by virtue of~\ref{fact-syst}.
\begin{lemma}\label{l:iota-sections}
$\iota_{(u,v)}$ is a section. 
\end{lemma}
\begin{proof}
Just observe that, since $L\co{v}\to u^\#$, there is $g\colon L \Sg{v} \to a$ such that ${g\circ L\co{v} = u^\#}$. By the universal property of the pushout, there is an arrow $h\colon a +_{L P} L \Sg{v}\to a$ such that $h\circ \iota_{(u,v)}$ is the identity of $a$. 
\end{proof}

We let $D$ be the diagram in $\E$ consisting of all the morphisms 
\[
\iota_{(u,v)}\colon a\to a +_{L P} L \Sg{v}
\]
as above. Because $\C_k$ is locally finite and $e$ varies among the objects of $L[\o{\C}_k]$, choosing representatives for isomorphism classes in an appropriate way we can assume by Remark~\ref{rem:finitely-many-finite-colim} that $D$ is a small diagram. 

By~\ref{lim-colim}, $D$ admits a colimit $b\coloneqq \colim D$. In other words, $b$ is obtained as a \emph{wide pushout} in $\E$. Denote by $s\colon a \to b$ the colimit map with domain $a$, and by 
\[
t_{(u,v)}\colon a +_{L P} L \Sg{v}\to b
\] 
the colimit map corresponding to the arrow $\iota_{(u,v)}$. As all arrows in $D$ are sections by Lemma~\ref{l:iota-sections}, so are the colimit maps. In particular, $s\colon a \to b$ is a section.

We claim that $b$ is $k$-extendable relative to $s$, thus settling Proposition~\ref{pr:relative-extension}. Assume we are given path embeddings $m\colon P\emb Ra$ and $n\colon P\emb Re$, with $e\in L[\o{\C}_k]$, such that $\co{m}\rl\co{n}$ as displayed in the leftmost diagram below.
\begin{equation*}
\begin{tikzcd}[column sep=2em]
{} & P \arrow[rightarrowtail]{dl}[swap]{\co{m}} \arrow[rightarrowtail]{dr}{\co{n}} & {} \\
\Sg{m} \arrow[yshift=3pt]{rr}{f} & & \Sg{n} \arrow[yshift=-3pt]{ll}{g}
\end{tikzcd}
\ \ \ \ \ \ \ 
\begin{tikzcd}[column sep=2em]
{} & L P \arrow[rightarrowtail]{dl}[swap]{m^{\#}} \arrow[rightarrowtail]{dr}{L \co{n}} & {} \\
a & & L \Sg{n} \arrow{ll}[swap]{\inc{m}^\#\circ Lg}
\end{tikzcd}
\end{equation*}
If $\inc{m}\colon \Sg{m}\emb Ra$ is the canonical embedding, we get a commutative triangle as on the right-hand side above. Just observe that
\[
\inc{m}^\#\circ Lg \circ L\co{n} = \inc{m}^\#\circ L\co{m} = (\inc{m} \circ \co{m})^\# = m^\#. 
\] 
Hence $L\co{n}\to m^\#$. Let 
\[
\iota_{(m, n)}\colon a \to a +_{L P} L \Sg{n}
\]
be the corresponding arrow in the diagram $D$. 
To improve readability we shall write, respectively, $\iota$, $\lambda$ and $t$ instead of $\iota_{(m, n)}$, $\lambda_{(m, n)}$ and $t_{(m, n)}$.

Let $n'\colon Q\emb R e$ be a path embedding with $n\leq n'$ in $\Path{(R e)}$. 
We must exhibit a path embedding $m'\colon Q\emb R b$ such that the leftmost square below commutes and $\co{m'}\rl \co{n'}$.
\begin{equation}\label{eq:two-prop-rel}
\begin{tikzcd}
P \arrow[rightarrowtail]{r}{!} \arrow[rightarrowtail]{d}[swap]{m} & Q \arrow[rightarrowtail,dashed]{d}{m'} \\
R a \arrow{r}{R s} & R b
\end{tikzcd}
\ \ \ \ \ \ \ 
\begin{tikzcd}[column sep=2em]
{} & Q \arrow[rightarrowtail]{dl}[swap]{\co{m'}} \arrow[rightarrowtail]{dr}{\co{n'}} & {} \\
\Sg{m'} \arrow[yshift=3pt,dashed]{rr} & & \Sg{n'} \arrow[yshift=-3pt,dashed]{ll}
\end{tikzcd}
\end{equation}
Note that, because $n\leq n'$, we get $\Sg{n'}\leq \Sg{n}$ in $\Emb{Re}$. Thus, $\co{n'}$ can be identified with a path embedding into $\Sg{n}$.
Consider the arrow 
\[
\xi\coloneqq (\lambda\circ L \co{n'})^\flat\colon Q \to R(a +_{L P} L\Sg{n}).
\]

\begin{lemma}\label{l:xi-emb}
$\xi$ is an embedding.
\end{lemma}
\begin{proof}
By the first part of~\ref{path-emb}, $(n')^\#$ is an embedding. Then $Ln'$ is an embedding since $(n')^\# = \epsilon_e \circ Ln'$, and so is $L\co{n'}$. It follows that $\lambda\circ L \co{n'}$ is an embedding because it is a composition of embeddings, and $\xi$ is an embedding by the second part of~\ref{path-emb}.
\end{proof}
Lemma~\ref{l:xi-emb}, combined with the fact that $Rt$ is a section (hence an embedding), entails that the composite $m'\coloneqq Rt\circ \xi \colon Q\emb R b$ is an embedding. Moreover
\begin{align*}
m'\circ {!} &= ((m'\circ {!})^{\#})^\flat = ((m')^\# \circ L{!})^\flat = (t\circ \lambda \circ L\co{n'} \circ L{!})^\flat \\
&= (t\circ \lambda \circ L\co{n})^\flat = (t\circ \iota\circ m^\#)^\flat = (s\circ m^\#)^\flat = Rs \circ m,
\end{align*}
showing that the leftmost diagram in equation~\eqref{eq:two-prop-rel} commutes.
Since $Rt$ is a section, we have $\xi\rl m'$ by Remark~\ref{rem:sections-coslice}, and so $\co{\xi}\rl \co{m'}$ by Lemma~\ref{l:corestriction-properties}\ref{corestriction-arrows}. Therefore, in order to show that $\co{m'}\rl \co{n'}$ it suffices to prove that $\co{\xi}\rl \co{n'}$. We have
\[
\lambda^\flat \circ \co{n'} = ((\lambda^\flat \circ \co{n'})^\#)^\flat = (\lambda \circ L\co{n'})^\flat = \xi
\]
and thus $\co{n'}\to \xi$. It follows from Remark~\ref{rem:idempotent} that $\co{n'}\to \co{\xi}$.

It remains to show that $\co{\xi}\to \co{n'}$; the proof of this fact will occupy us for the rest of this section.
As $\C_k$ is an arboreal category, $\Sg{\xi}$ is the colimit of its path embeddings. Thus, in order to define a morphism $\Sg{\xi} \to \Sg{n'}$ it suffices to define a compatible cocone with vertex $\Sg{n'}$ over the diagram of path embeddings into $\Sg{\xi}$. By Lemma~\ref{l:corestriction-properties}\ref{comparable-subtree}, the path embeddings into $\Sg{\xi}$ can be identified with the path embeddings into $R(a +_{L P} L\Sg{n})$ that are comparable with~$\xi$. For each such path embedding $q\colon Q'\emb R(a +_{L P} L\Sg{n})$, we shall define an arrow $\zeta_q\colon Q'\to Re$ and prove that these form a compatible cocone.
We will then deduce, using the induced mediating morphism $\Sg{\xi} \to Re$, that $\co{\xi}\to \co{n'}$.

Fix an arbitrary path embedding $q\colon Q'\emb R(a +_{L P} L\Sg{n})$ above $\xi$ and consider the following diagram in $\E$, where the four vertical faces are pullbacks.
\begin{equation}\label{eq:cube-of-embeddings}
\begin{tikzcd}[row sep=1em, column sep=2.5em]
\o{L P} \arrow[rr,rightarrowtail,"\nu"] \arrow[dr,swap,rightarrowtail,"\mu_1"] \arrow[dd,rightarrowtail,"\mu_2",swap] &&
\o{L\Sg{n}} \arrow[dd,rightarrowtail,"\tau_2",near end] \arrow[dr,rightarrowtail,"\tau_1"] \\
& \o{a} \arrow[rr,crossing over,rightarrowtail,"\sigma_1" near start] &&
LQ'  \\
LP \arrow[rr,rightarrowtail,"L\co{n}", near end] \arrow[dr,rightarrowtail,"m^\#",swap]  && L\Sg{n} \arrow[dr,rightarrowtail,"\lambda"] \\
& a \arrow[rr,rightarrowtail,"\iota"] \arrow[uu,leftarrowtail,crossing over,"\sigma_2", near end] & & a +_{L P} L\Sg{n} \arrow[uu,swap,leftarrowtail,"q^\#"]
\end{tikzcd}
\end{equation}
Note that the previous diagram consists entirely of embeddings because $q^\#$ is an embedding by the first part of~\ref{path-emb}, and the pullback in~$\E$ of an embedding exists by~\ref{lim-colim} and is again an embedding.

Because $q$ is above $\xi$, there is an embedding $Q\emb Q'$, and thus also an embedding $P\emb Q'$. By the universal property of pullbacks, there are unique arrows 
\[
\theta\colon LP\emb \o{a} \ \ \text{ and } \ \  \Delta\colon LP \emb \o{LP}
\] 
making the following diagrams commute.
\begin{equation*}
\begin{tikzcd}[column sep = 3em]
LP \arrow[rr, relay arrow=2ex, rightarrowtail, "L!"] 
\arrow[bend right=30,rightarrowtail]{dr}[swap, description]{m^{\#}} \arrow[dashed,rightarrowtail]{r}{\theta} & \o{a} \arrow[dr, phantom, "\lrcorner", very near start] \arrow[rightarrowtail]{r}{\sigma_1} \arrow[rightarrowtail]{d}[swap]{\sigma_2} & LQ' \arrow[rightarrowtail]{d}{q^\#} \\
{} & a \arrow[rightarrowtail]{r}{\iota} & {a +_{L P} L\Sg{n}}
\end{tikzcd}
\ \ \ \ \ \ \ 
\begin{tikzcd}[column sep = 3em]
LP \arrow[rr, relay arrow=2ex, "\theta", rightarrowtail] 
\arrow[bend right=30]{dr}[swap, description]{\id_{LP}} \arrow[dashed,rightarrowtail]{r}{\Delta} & \o{LP} \arrow[dr, phantom, "\lrcorner", very near start] \arrow[rightarrowtail]{r}{\mu_1} \arrow[rightarrowtail]{d}[swap]{\mu_2} & \o{a} \arrow[rightarrowtail]{d}{\sigma_2} \\
{} & LP \arrow[rightarrowtail]{r}{m^{\#}} & a
\end{tikzcd}
\end{equation*}
Note in particular that $\mu_2$ is a retraction whose right inverse is $\Delta$. As $\mu_2$ is also an embedding, it must be an isomorphism with (two-sided) inverse $\Delta$.

By~\ref{path-restriction-prop} (more precisely, by item~\ref{path-re-1} in Definition~\ref{def:path-re}) there is an arrow $w\colon P\to Q_{\o{a}}$ between paths such that $Lw = \theta$. Hence, we can consider $\sigma_2^\flat\colon Q_{\o{a}}\to Ra$. Note that
\[
(\sigma_2^\flat \circ w)^\# = \sigma_2 \circ Lw = \sigma_2 \circ \theta = m^\#,
\]  
and so $\sigma_2^\flat \circ w = m$. In particular, $w$ is an embedding. As $Q_{\o{a}}$ is a path, $\inc{w}\colon \Sg{w}\emb Q_{\o{a}}$ can be identified with the identity $Q_{\o{a}} \to Q_{\o{a}}$. By Lemma~\ref{l:corestriction-properties}\ref{corestriction-arrows-finer}, there is a unique arrow $\psi_q\colon Q_{\o{a}}\to \Sg{m}$ making the following diagram commute.
\[\begin{tikzcd}
Q_{\o{a}} \arrow[dashed]{rr}{\psi_q} \arrow[rightarrowtail]{dr}[swap]{\sigma_2^\flat} & & \Sg{m} \arrow[rightarrowtail]{dl}{\inc{m}} \\
{} & Ra & {}
\end{tikzcd}\]

\begin{lemma}\label{l:gamma-tilde-triangle}
The following diagram commutes.
\[\begin{tikzcd}
{} & \o{LP} \arrow[rightarrowtail]{dl}[swap]{\mu_1} \arrow[rightarrowtail]{dr}{L\co{n}\circ \mu_2} & {} \\
\o{a} \arrow{rr}{L(f\circ \psi_q)} & & L\Sg{n}
\end{tikzcd}\]
\end{lemma}
\begin{proof}
Note that 
\[\inc{m}\circ \psi_q\circ w = \sigma_2^\flat \circ w = m = \inc{m}\circ \co{m}
\]
and so $\psi_q\circ w = \co{m}$ since $\inc{m}$ is a monomorphism.
Applying the functor $L$ to the outer commutative diagram on the left-hand side below, we obtain the commutative diagram on the right-hand side.
\begin{center}
\begin{tikzcd}[row sep = 2.5em, column sep = 2.5em]
{} & P \arrow[rightarrowtail]{dl}[swap]{w} \arrow[rightarrowtail]{d}[description]{\co{m}} \arrow[rightarrowtail]{dr}{\co{n}} & {} \\
Q_{\o{a}} \arrow{r}{\psi_q} & \Sg{m} \arrow{r}{f} & \Sg{n}
\end{tikzcd}
\ \ \ \ \ \ \ 
\begin{tikzcd}[row sep = 2.5em]
{} & LP \arrow[rightarrowtail]{dl}[swap]{\theta}  \arrow[rightarrowtail]{dr}{L\co{n}} & {} \\
\o{a} \arrow{rr}{L(f\circ \psi_q)} & & L\Sg{n}
\end{tikzcd}
\end{center}
Hence, precomposing with $\mu_2$ we get
\[
L\co{n} \circ \mu_2 = L(f\circ \psi_q) \circ \theta \circ \mu_2 = L(f\circ \psi_q) \circ \mu_1. \qedhere
\]
\end{proof}

For convenience of notation, let us write 
\[
\tilde{\gamma}_q\coloneqq L(f\circ \psi_q)\colon \o{a}\to L\Sg{n} \ \text{ and } \ \gamma_q \coloneqq \epsilon_{e}\circ L\inc{n}\circ \tilde{\gamma}_q\colon \o{a}\to e.
\]
With this notation we have
\begin{align*}
\gamma_q \circ \mu_1 &= \epsilon_{e}\circ L\inc{n}\circ \tilde{\gamma}_q \circ \mu_1 \\
&= \epsilon_{e}\circ L\inc{n}\circ L\co{n}\circ \mu_2 \tag*{Lemma~\ref{l:gamma-tilde-triangle}} \\
&= \epsilon_{e}\circ L\inc{n}\circ \tau_2 \circ \nu
\end{align*}
and so the leftmost diagram below commutes.
\begin{center}
\begin{tikzcd}
\o{L P} \arrow[rightarrowtail]{r}{\nu} \arrow[rightarrowtail]{d}[swap]{\mu_1} & \o{L\Sg{n}} \arrow[rightarrowtail]{d}{\epsilon_{e}\circ L\inc{n}\circ \tau_2} \\
\o{a} \arrow{r}{\gamma_q} & e
\end{tikzcd}
\ \ \ \ \ \ \ 
\begin{tikzcd}
\o{L P} \arrow[rightarrowtail]{r}{\nu} \arrow[rightarrowtail]{d}[swap]{\mu_1} & \o{L\Sg{n}} \arrow[rightarrowtail]{d}{\tau_1} \\
\o{a} \arrow[rightarrowtail]{r}{\sigma_1} & L Q' \arrow[ul, phantom, "\ulcorner", very near start]
\end{tikzcd}
\end{center}
Now, note that by~\ref{fact-syst} the top face of diagram~\eqref{eq:cube-of-embeddings}, displayed in the rightmost diagram above, is a pushout in $\E$.
By the universal property of pushouts, there is  a unique $\delta_q\colon LQ' \to e$ satisfying 
\[
\delta_q\circ \sigma_1=\gamma_q \ \text{ and } \ \delta_q\circ \tau_1= \epsilon_{e}\circ L\inc{n} \circ \tau_2.
\]

Define $\zeta_q\coloneqq (\delta_q)^\flat\colon Q'\to Re$ for all path embeddings $q\colon Q'\emb R(a +_{L P} L\Sg{n})$ above $\xi$. Further, if $q$ is below $\xi$, we let $\zeta_q$ be the obvious restriction of $\zeta_{\xi}$. 

\begin{lemma}
The following family of arrows forms a compatible cocone over the diagram of path embeddings into $\Sg{\xi}$:
\[
\{\zeta_q \mid q\colon Q'\emb R(a +_{L P} L\Sg{n}) \ \text{is a path embedding comparable with $\xi$}\}.
\] 
\end{lemma}
\begin{proof}
Fix arbitrary path embeddings 
\[
q\colon Q'\emb R(a +_{L P} L\Sg{n}) \ \text{ and } \ q'\colon Q''\emb R(a +_{L P} L\Sg{n})
\] 
comparable with $\xi$.
The compatibility condition for the cocone states that $\zeta_q$ extends $\zeta_{q'}$ whenever $q\geq q'$. It suffices to settle the case where $\xi\leq q\leq q'$. Also, it is enough to show that $\delta_q$ extends $\delta_{q'}$, \ie $\delta_q \circ L{!} = \delta_{q'}$ where ${!}\colon Q''\emb Q'$ is the unique embedding. Just observe that $\delta_q \circ L{!} = \delta_{q'}$ entails
\[
\zeta_q\circ {!} = \delta_q^\flat \circ {!} = ((\delta_q^\flat \circ {!})^\#)^\flat = (\delta_q \circ L{!})^\flat = \delta_{q'}^\flat = \zeta_{q'}.
\]
Consider the following diagram all whose vertical faces are pullbacks and note that by~\ref{path-restriction-prop}, and more precisely by item~\ref{path-re-2} in Definition~\ref{def:path-re}, the pullback of $L!$ along $\sigma_1$ is of the form $L\ell$ for some arrow $\ell\colon Q_{\o{\o{a}}}\to  Q_{\o{a}}$. 
\[\begin{tikzcd}[row sep=1em, column sep=2.5em]
\o{\o{L P}} \arrow[rr,rightarrowtail,""] \arrow[dr,swap,rightarrowtail,""] \arrow[dd,rightarrowtail,"",swap] &&
\o{\o{L\Sg{n}}} \arrow[dd,rightarrowtail,"\o{\tau}_2",near end] \arrow[dr,rightarrowtail,"\o{\tau}_1"]  &&\\
& \o{\o{a}} \arrow[rr,crossing over,rightarrowtail,"\o{\sigma}_1", near start] &&
LQ'' \arrow[dd,rightarrowtail,"L{!}"] \arrow[dddr,"\delta_{q'}", bend left = 30] && \\
\o{L P} \arrow[rr,rightarrowtail,"\nu", near end] \arrow[dr,swap,rightarrowtail,"\mu_1"] \arrow[dd,rightarrowtail,"\mu_2",swap] &&
\o{L\Sg{n}} \arrow[dd,rightarrowtail,"\tau_2",near end] \arrow[dr,rightarrowtail,"\tau_1"] && \\
& \o{a} \arrow[rr,crossing over,rightarrowtail,"\sigma_1" near start] \arrow[uu,leftarrowtail,crossing over,"L\ell", near end] &&
LQ' \arrow[dr,"\delta_q"] && \\
LP \arrow[rr,rightarrowtail,"L\co{n}", near end] \arrow[dr,rightarrowtail,"m^\#",swap]  && L\Sg{n} \arrow[dr,rightarrowtail,"\lambda"]  & & e \\
& a \arrow[rr,rightarrowtail,"\iota"] \arrow[uu,leftarrowtail,crossing over,"\sigma_2", near end] & & a +_{L P} L\Sg{n} \arrow[uu,swap,leftarrowtail,"q^\#"] &&
\end{tikzcd}\]
In view of the definition of $\delta_{q'}$ in terms of the universal property of pushouts, it suffices to show that $\delta_q \circ L{!}$ satisfies 
\[
(\delta_q \circ L{!})\circ \o{\sigma}_1 = \gamma_{q'} \ \text{ and } \ (\delta_q \circ L{!}) \circ \o{\tau}_1= \epsilon_{e}\circ L\inc{n} \circ \tau_2\circ \o{\tau}_2.
\]
The latter equation follows at once from the identity $\delta_q \circ \tau_1 = \epsilon_{e} \circ L\inc{n} \circ \tau_2$. With regards to the former, we have
\[
(\delta_q \circ L{!})\circ \o{\sigma}_1 = \delta_q \circ \sigma_1 \circ L\ell = \gamma_q \circ L\ell.
\]
Thus it suffices to show that $\gamma_q \circ L\ell = \gamma_{q'}$, and this clearly follows if we prove that $\tilde{\gamma}_q \circ L\ell = \tilde{\gamma}_{q'}$. 
Recall that $\psi_{q'}$ is the unique morphism such that the composite
\[\begin{tikzcd}
Q_{\o{\o{a}}} \arrow{r}{\psi_{q'}} & \Sg{m} \arrow[rightarrowtail]{r}{\inc{m}} & Ra
\end{tikzcd}\]
coincides with $(\sigma_2 \circ L\ell)^\flat$. But
\[
(\sigma_2 \circ L\ell)^\flat = ((\sigma_2^\flat \circ \ell)^\#)^\flat = \sigma_2^\flat \circ \ell,
\]
so $\inc{m}\circ \psi_{q}\circ \ell = \sigma_2^\flat \circ \ell$ entails $\psi_{q'} = \psi_q \circ \ell$. Therefore,
\[
\tilde{\gamma}_{q'} = L(f\circ \psi_{q'}) = L(f\circ \psi_q \circ \ell) = \tilde{\gamma}_{q}\circ L\ell. \qedhere
\]
\end{proof}

The previous lemma entails the existence of a unique morphism $h\colon \Sg{\xi} \to Re$ satisfying $h\circ q = \zeta_q$ for all path embeddings $q$ into $R(a +_{L P} L\Sg{n})$ that are comparable with $\xi$. In order to conclude that $\co{\xi} \to \co{n'}$ as desired, we prove the following useful property of the cocone consisting of the morphisms $\zeta_q$.

\begin{lemma}
Let $G\coloneqq LR$ and consider the composite morphism
\[\begin{tikzcd}
RL\Sg{n} \arrow{r}{RL \inc{n}} & RGe \arrow{r}{R\epsilon_e} & Re.
\end{tikzcd}\] 
If $q= R\lambda\circ \alpha$ for some arrow $\alpha\colon Q'\emb RL\Sg{n}$, then $\zeta_q = R\epsilon_e \circ RL \inc{n} \circ \alpha$.
\end{lemma}
\begin{proof}
Suppose that $q= R\lambda\circ \alpha$ for some $\alpha\colon Q'\emb RL\Sg{n}$. 
We have 
\[
(R\epsilon_e \circ RL \inc{n} \circ \alpha)^\# = ((\epsilon_e \circ L \inc{n} \circ \alpha^\#)^\flat)^\# = \epsilon_e \circ L \inc{n} \circ \alpha^\#.
\]
By the universal property of $\delta_q = \zeta_q^\#$, $\zeta_q = R\epsilon_e \circ RL \inc{n} \circ \alpha$ if, and only if,
\begin{equation}\label{eq:cond-1}
(\epsilon_e \circ L \inc{n} \circ \alpha^\#) \circ \tau_1 = \epsilon_{e}\circ L\inc{n} \circ \tau_2
\end{equation}
and 
\begin{equation}\label{eq:cond-2}
(\epsilon_e \circ L \inc{n} \circ \alpha^\#) \circ \sigma_1 = \epsilon_{e}\circ L\inc{n} \circ \tilde{\gamma}_q.
\end{equation}
Observe that $\alpha^\# \circ \tau_1 = \tau_2$ because
\[
\lambda\circ \alpha^\# \circ \tau_1 = q^\# \circ \tau_1 = \lambda \circ \tau_2
\]
and $\lambda$ is a monomorphism. Thus, equation~\eqref{eq:cond-1} holds.
Further, note that
\[
q^\# = (R\lambda\circ \alpha)^\# = ((\lambda\circ \alpha^\#)^\flat)^\# = \lambda\circ \alpha^\#
\]
and so $\tau_1$ in diagram~\eqref{eq:cube-of-embeddings} is an isomorphism. As pushout squares of embeddings in $\E$ are also pullbacks by~\ref{fact-syst}, $\mu_1$ in diagram~\eqref{eq:cube-of-embeddings} is also an isomorphism.
Therefore,
\begin{align*}
\lambda \circ \alpha^\# \circ \sigma_1 &= q^\# \circ \sigma_1 = q^\# \circ \sigma_1 \circ \mu_1 \circ \mu_1^{-1} = \lambda\circ L\co{n} \circ \mu_2 \circ \mu_1^{-1} \\
& = \lambda\circ \tilde{\gamma}_q \circ \mu_1 \circ \mu_1^{-1} = \lambda\circ \tilde{\gamma}_q
\end{align*}
and so $\alpha^\# \circ \sigma_1 = \tilde{\gamma}_q$. Equation~\eqref{eq:cond-2} then follows at once.
\end{proof}

Since $\xi = R\lambda \circ (L\co{n'})^\flat$, recalling that we identify $\co{n'}$ with a path embedding into $\Sg{n}$, an application of the previous lemma with $\alpha\coloneqq (L\co{n'})^\flat$ yields
\[
\zeta_{\xi} = R\epsilon_e \circ RL \inc{n} \circ (L\co{n'})^\flat = (\epsilon_e \circ L\inc{n} \circ L\co{n'})^\flat = (\epsilon_e \circ L n')^\flat = ((n')^\#)^\flat = n'.
\]
In other words, $h\circ \co{\xi}=n'$ and so $\co{\xi}\to n'$. It follows from Remark~\ref{rem:idempotent} that $\co{\xi} \to \co{n'}$, thus concluding the proof of Proposition~\ref{pr:relative-extension}. 

\section{Relativisations to classes of well-behaved structures}
\label{s:well-behaved}
%
In this section, we discuss a family of simple relativisation results which are, in a sense, oblivious to the ``tame'' versus ``not-so-tame'' divide of the landscape. Recall from Section~\ref{s:exploring-the-landscape} that, typically, the essence of a homomorphism preservation theorem is captured by the following configuration:
\[\begin{tikzcd}[column sep=0.1em]
a & {\eqaCk} & G_k a \arrow{rrrrrrrrrrrrrrr}{} &&&&& &&&&& &&&&&  b.
\end{tikzcd}\] 
The aim consists in proving that, if $a$ belongs to a full subcategory $\D$ of the extensional category $\E$, then so does $b$. The category $\D$ is assumed to be closed in $\E$ under morphisms, and saturated under the relation $\eqbCk$. 

The results presented in this paper are based on \emph{upgrading arguments} (see~\cite{Otto2011} for an overview of this method): the relation $\eqaCk$ between $a$ and $G_{k}a$ is upgraded to the finer relation $\eqbCk$, possibly between appropriate extensions $a^{*}$ and $(G_{k}a)^{*}$. However, if $a$ admits an arrow to $G_{k}a$, and thus also to $b$, then no upgrading is needed; in this case, the assumption that $\D$ be saturated under $\eqbCk$ is superfluous. The existence of a morphism $a\to G_{k}a$ is a strong requirement but, as we shall see below, it is satisfied in several cases of interest.

Fix a finite relational vocabulary $\sg$ and recall that every finite structure in $\CS$ admits a \emph{core}, denoted by $\core(A)$, which is unique up to isomorphism. This can be characterised as the minimal substructure of $A$ (with respect to set-theoretic inclusion) onto which $A$ retracts, and any two homomorphically equivalent finite $\sg$-structures have isomorphic cores; see \eg \cite{HN1992}. For the next lemma, consider any comonadic adjunction 
\[
L \dashv R\colon \CS\to \C.
\]
We let $G\coloneqq LR$ be the induced comonad, and denote by $L[\C]$ the full subcategory of $\CS$ on those objects that are in the image of the functor $L$.

\begin{lemma}\label{l:core}
If $L[\C]$ is closed in $\CS$ under induced substructures, then the following statements are equivalent for all $A\in \CS$:
\begin{enumerate}
\item\label{i:A-to-GA} There exists a homomorphism $A\to GA$.
\item\label{i:hom-equiv} $A$ is homomorphically equivalent to a structure $A'\in L[\C]$.
\end{enumerate}
Moreover, if $A$ is finite, the previous conditions are equivalent to the following:
\begin{enumerate}[resume]
\item\label{i:core} $\core(A)\in L[\C]$.
\end{enumerate}
\end{lemma}
\begin{proof}
\ref{i:A-to-GA} $\Rightarrow$ \ref{i:hom-equiv}. Suppose there is a homomorphism $A \to GA$, and decompose it as
\[
A \twoheadrightarrow A' \hookrightarrow GA
\]
where $A'$ is an induced substructure of $GA$. Composing the inclusion $A' \hookrightarrow GA$ with the counit $\epsilon\colon GA \to A$, we get a homomorphism $A'\to A$. It follows that $A$ and $A'$ are homomorphically equivalent. Note that $A'$ belongs to $L[\C]$ because it is an induced substructure of $GA$, which belongs to $L[\C]$.

\ref{i:hom-equiv} $\Rightarrow$ \ref{i:A-to-GA}. If $A$ is homomorphically equivalent to some $A'\in L[\C]$ then, by functoriality, $GA$ is homomorphically equivalent to $GA'$. Since the adjunction $L \dashv R$ is comonadic and $A'\in L[\C]$, there is a coalgebra map $A'\to GA'$. We thus obtain a diagram
\[\begin{tikzcd}
A \arrow[yshift=3pt]{r} \arrow[leftarrow, yshift=-3pt]{r} & A' \arrow{d} \\
GA \arrow[yshift=3pt]{r} \arrow[leftarrow, yshift=-3pt]{r} & GA'
\end{tikzcd}\]
and so $A\to GA$.

\ref{i:hom-equiv} $\Rightarrow$ \ref{i:core}. Suppose that $A$ is homomorphically equivalent to some $A'\in L[\C]$. Let $B$ be the image of $A$ under any homomorphism $A\to A'$ and note that $B$ is finite, because so is $A$, and it is homomorphically equivalent to $A$. Further, $B$ belongs to $L[\C]$ because the latter contains $A'$ and is closed under induced substructures. Thus,
\[
\core(A)\cong \core(B) \hookrightarrow B \in L[\C]
\]
entails $\core(A) \in L[\C]$.

\ref{i:core} $\Rightarrow$ \ref{i:hom-equiv}. Take $A'\coloneqq \core(A)$.
\end{proof}

Note that the equivalence between items~\ref{i:A-to-GA} and~\ref{i:hom-equiv} in Lemma~\ref{l:core} holds, more generally, if we replace the category of $\sg$-structures with any category $\E$ equipped with a class $\M$ of embeddings, assuming that $L[\C]$ is closed under taking $\M$-subobjects. For simplicity, we opted to state this lemma, as well as the following proposition, in the particular case where $\E=\CS$.

Consider any resource-indexed arboreal cover of $\CS$ by $\C$, with adjunctions ${L_{k} \dashv R_{k}}$, such that each category $L_{k}[\C_{k}]$ is closed in $\CS$ under induced substructures. We obtain the following abstract homomorphism preservation theorem relative to a class of structures $\E'$:
\begin{proposition}\label{p:abstract-HPT-wellbehaved}
 Let $\E'$ be a full subcategory of $\CS$, and suppose there is~$k$ such that every object of $\E'$ is homomorphically equivalent to an object of $L_{k}[\C_{k}]$. For any full subcategory $\D$ of $\CS$, $\D \cap \E'$ is closed under morphisms in $\E'$ precisely when it is upwards closed in $\E'$ with respect to $\prCk$.
\end{proposition}
\begin{proof}
For the non-trivial direction, assume $\D \cap \E'$ is closed under morphisms in~$\E'$. We must show that, for all $A, B\in \E'$ such that $A\prCk B$, if $A\in \D$ then also $B\in \D$. By Lemma~\ref{l:core}, there is a homomorphism $A\to G_{k}A$, where $G_k\coloneqq L_k R_k$. Further, there is a homomorphism $G_{k}A\to B$ because $A\prCk B$. The composite yields a homomorphism $A\to B$, and so $B\in \D$.
\end{proof}

Let us say that a class of $\sg$-structures has \emph{tree-depth at most} $k$ if each of its members has tree-depth at most $k$. Consider the Ehrenfeucht-\Fraisse~resource-indexed arboreal cover of $\CS$ by $\RT(\sg)$ from Example~\ref{ex:res-ind-arb-cover}. In view of \cite[Theorem~6.2]{AS2021}, the category $L_{k}[\RTk(\sg)]$ can be identified with the class of $\sg$-structures having tree-depth at most $k$. If $A$ is a $\sg$-structure of tree-depth at most $k$, then any induced substructure of $A$ has also tree-depth at most $k$; hence, $L_{k}[\RTk(\sg)]$ is closed in $\CS$ under induced substructures. Proposition~\ref{p:abstract-HPT-wellbehaved} then implies that, on any class of $\sg$-structures with tree-depth at most $k$, any first-order sentence that is preserved under homomorphisms is equivalent to an existential positive sentence of quantifier rank at most $k$.

\begin{remark}
It is enough to consider the Ehrenfeucht-\Fraisse~resource-indexed arboreal cover of $\CS$ by $\RT(\sg)$, rather than the resource-indexed arboreal adjunction between $\CS$ and $\RT(\sg^I)$ in Example~\ref{logical-rel-EF}, because every existential positive sentence (with quantifier rank at most $k$) is equivalent to one (with quantifier rank at most $k$) that does not use the equality symbol.
\end{remark}

In a similar fashion, we deduce the following result for classes of finite $\sg$-structures:
\begin{corollary}\label{cor:finite-HPT-bounded-tree-depth}
On any class of finite $\sg$-structures whose cores have tree-depth at most~$k$, every first-order sentence that is preserved under homomorphisms is also definable by an existential positive sentence of quantifier rank at most $k$.
\end{corollary}

Note that these are not equi-resource homomorphism preservation theorems, as the quantifier rank of the ensuing existential positive sentence is bounded by the tree-depth of the class of (cores of) structures under consideration, which may be larger than the quantifier rank of the first-order sentence we started with. 

On the other hand, since we did not assume in Proposition~\ref{p:abstract-HPT-wellbehaved} that $\D$ is saturated under $\eqbCk$, Corollary~\ref{cor:finite-HPT-bounded-tree-depth} holds, more generally, for all (not necessarily first-order definable) Boolean queries; see \eg \cite[\S 2.1]{Libkin2004} for the latter notion. Generalisations to non-Boolean queries are also available, by considering extensions of the Ehrenfeucht-\Fraisse~resource-indexed arboreal cover to categories of $n$-pointed $\sg$-structures.

\section*{Acknowledgment}
We are grateful to the anonymous referees for the numerous comments and suggestions that helped us to improve the presentation of our results.


\bibliographystyle{amsplain-nodash}

\begin{thebibliography}{10}

\bibitem{ABRAMSKY2020}
S.~Abramsky, \emph{Whither semantics?}, Theoretical Computer Science
  \textbf{807} (2020), 3--14.

\bibitem{emerging2022}
S.~Abramsky, \emph{{S}tructure and {P}ower: an emerging landscape}, Fundamenta
  Informaticae \textbf{186} (2022), no.~1--4, pp.~1--26, Special Issue in Honor
  of the Boris Trakhtenbrot Centenary.

\bibitem{abramsky2017pebbling}
S.~Abramsky, A.~Dawar, and P.~Wang, \emph{The pebbling comonad in finite model
  theory}, 32nd Annual {ACM/IEEE} Symposium on Logic in Computer Science
  (LICS), 2017, pp.~1--12.

\bibitem{Guarded2021}
S.~Abramsky and D.~Marsden, \emph{Comonadic semantics for guarded fragments},
  Proceedings of the 36th Annual {ACM/IEEE} Symposium on Logic in Computer
  Science, {LICS}, 2021.

\bibitem{Hybrid2022}
S.~Abramsky and D.~Marsden, \emph{Comonadic semantics for hybrid logic}, 47th
  International Symposium on Mathematical Foundations of Computer Science
  (MFCS), Leibniz International Proceedings in Informatics, vol. 241, Schloss
  Dagstuhl -- Leibniz-Zentrum f{\"u}r Informatik, 2022, pp.~7:1--7:14.

\bibitem{AR2021icalp}
S.~Abramsky and L.~Reggio, \emph{Arboreal categories and resources}, 48th
  International Colloquium on Automata, Languages, and Programming (ICALP
  2021), Leibniz International Proceedings in Informatics, vol. 198, Schloss
  Dagstuhl -- Leibniz-Zentrum f{\"u}r Informatik, 2021, pp.~115:1--115:20.

\bibitem{AR2022}
S.~Abramsky and L.~Reggio, \emph{Arboreal categories: An axiomatic theory of
  resources}, Logical Methods in Computer Science \textbf{19} (2023), no.~3,
  14:1--14:36.

\bibitem{DBLP:conf/csl/AbramskyS18}
S.~Abramsky and N.~Shah, \emph{Relating structure and power: Comonadic
  semantics for computational resources}, 27th {EACSL} Annual Conference on
  Computer Science Logic ({CSL}), 2018, pp.~2:1--2:17.

\bibitem{AS2021}
S.~Abramsky and N.~Shah, \emph{Relating structure and power: Comonadic
  semantics for computational resources}, Journal of Logic and Computation
  \textbf{31} (2021), no.~6, 1390--1428.

\bibitem{adamek2004abstract}
J.~Ad{\'a}mek, H.~Herrlich, and G.~E. Strecker, \emph{{Abstract and concrete
  categories. The joy of cats}}, Online edition, 2004.

\bibitem{AR1994}
J.~Ad\'{a}mek and J.~Rosick\'{y}, \emph{Locally presentable and accessible
  categories}, London Mathematical Society Lecture Note Series, vol. 189,
  Cambridge University Press, Cambridge, 1994.

\bibitem{HNvB1998}
H.~Andr{\'e}ka, J.~{\noopsort{Benthem}}{van Benthem}, and I.~N{\'e}meti,
  \emph{Modal languages and bounded fragments of predicate logic}, Journal of
  Philosophical Logic \textbf{27} (1998), no.~3, 217--274.

\bibitem{Barwise1977}
J.~Barwise, \emph{On {M}oschovakis closure ordinals}, Journal of Symbolic Logic
  \textbf{42} (1977), no.~2, 292--296.

\bibitem{beke2012abstract}
T.~Beke and J.~Rosick\'y, \emph{Abstract elementary classes and accessible
  categories}, Annals of Pure and Applied Logic \textbf{163} (2012), no.~12,
  2008--2017.

\bibitem{vBpieces}
J.~{\noopsort{Benthem}}{van Benthem}, \emph{Dynamic bits and pieces}, ILLC
  research report, University of Amsterdam, 1997.

\bibitem{blackburn2002modal}
P.~Blackburn, M.~De~Rijke, and Y.~Venema, \emph{Modal logic}, Cambridge Tracts
  in Theoretical Computer Science, vol.~53, Cambridge University Press, 2002.

\bibitem{conghaile2021game}
A.~{\noopsort{Conghaile}}{\'{O} Conghaile} and A.~Dawar, \emph{Game comonads
  {\&} generalised quantifiers}, 29th {EACSL} Annual Conference on Computer
  Science Logic, {CSL}, LIPIcs, vol. 183, Schloss Dagstuhl - Leibniz-Zentrum
  f{\"{u}}r Informatik, 2021, pp.~16:1--16:17.

\bibitem{DP2002}
B.~A. Davey and H.~A. Priestley, \emph{Introduction to lattices and order},
  second ed., Cambridge University Press, New York, 2002.

\bibitem{DJR2021}
A.~Dawar, T.~Jakl, and L.~Reggio, \emph{Lov{\'a}sz-type theorems and game
  comonads}, Proceedings of the 36th Annual {ACM/IEEE} Symposium on Logic in
  Computer Science, {LICS}, 2021.

\bibitem{deRijke2000}
M.~de~Rijke, \emph{A note on graded modal logic}, Studia Logica \textbf{64}
  (2000), no.~2, 271--283.

\bibitem{Ehrenfeucht1960}
A.~Ehrenfeucht, \emph{An application of games to the completeness problem for
  formalized theories}, Fund. Math. \textbf{49} (1960/61), 129--141.

\bibitem{Fraisse1954}
R.~Fra\"{\i}ss\'{e}, \emph{Sur quelques classifications des syst\`emes de
  relations}, Publ. Sci. Univ. Alger. S\'{e}r. A \textbf{1} (1954), 35--182.

\bibitem{freyd1972categories}
P.~J. Freyd and G.~M. Kelly, \emph{Categories of continuous functors, {I}},
  Journal of Pure and Applied Algebra \textbf{2} (1972), no.~3, 169--191.

\bibitem{HN1992}
P.~Hell and J.~Ne\v{s}et\v{r}il, \emph{The core of a graph}, Discrete
  Mathematics \textbf{109} (1992), no.~1, 117--126.

\bibitem{HM1980}
M.~Hennessy and R.~Milner, \emph{On observing nondeterminism and concurrency},
  Automata, Languages and Programming (Berlin, Heidelberg), Springer Berlin
  Heidelberg, 1980, pp.~299--309.

\bibitem{Hod93}
W.~Hodges, \emph{Model theory}, Encyclopedia of Mathematics and its
  Applications, vol.~42, Cambridge University Press, 1993.

\bibitem{Immerman1982}
N.~Immerman, \emph{Upper and lower bounds for first order expressibility},
  Journal of Computer and System Sciences \textbf{25} (1982), no.~1, 76--98.

\bibitem{FVM2022}
T.~Jakl, D.~Marsden, and N.~Shah, \emph{A game comonadic account of {C}ourcelle
  and {F}eferman-{V}aught-{M}ostowski theorems}, Preprint available at
  \url{https://arxiv.org/abs/2205.05387}, 2022.

\bibitem{JM1994}
A.~Joyal and I.~Moerdijk, \emph{A completeness theorem for open maps}, Ann.
  Pure Appl. Logic \textbf{70} (1994), no.~1, 51--86.

\bibitem{JNW1993}
A.~Joyal, M.~Nielsen, and G.~Winskel, \emph{Bisimulation and open maps},
  Proceedings of 8th Annual IEEE Symposium on Logic in Computer Science, 1993,
  pp.~418--427.

\bibitem{Adhesive2004}
S.~Lack and P.~Soboci{\'{n}}ski, \emph{Adhesive categories}, Foundations of
  Software Science and Computation Structures (Berlin, Heidelberg), Springer
  Berlin Heidelberg, 2004, pp.~273--288.

\bibitem{Libkin2004}
L.~Libkin, \emph{Elements of finite model theory}, Texts in Theoretical
  Computer Science. An EATCS Series, Springer, 2004.

\bibitem{Los1955}
J.~{\L}o{\'s}, \emph{On the extending of models. {I}}, Fund. Math. \textbf{42}
  (1955), 38--54.

\bibitem{Lyndon1959}
R.~C. Lyndon, \emph{Properties preserved under homomorphism}, Pacific J. Math.
  \textbf{9} (1959), 143--154.

\bibitem{MacLane}
S.~Mac~Lane, \emph{Categories for the working mathematician}, 2nd ed., GTM,
  vol.~5, Springer-Verlag, New York, 1998.

\bibitem{nevsetvril2006tree}
J.~Ne{\v{s}}et{\v{r}}il and P.~Ossona~de Mendez, \emph{Tree-depth, subgraph
  coloring and homomorphism bounds}, European Journal of Combinatorics
  \textbf{27} (2006), no.~6, 1022--1041.

\bibitem{Otto2011}
M.~Otto, \emph{Model theoretic methods for fragments of {FO} and special
  classes of (finite) structures}, Finite and Algorithmic Model Theory
  (J.~Esparza, C.~Michaux, and C.~Steinhorn, eds.), London Mathematical Society
  Lecture Note Series, vol. 379, Cambridge University Press, 2011,
  pp.~271--341.

\bibitem{Paine2020}
T.~Paine, \emph{A pebbling comonad for finite rank and variable logic, and an
  application to the equirank-variable homomorphism preservation theorem},
  Electronic Notes in Theoretical Computer Science \textbf{352} (2020),
  191--209.

\bibitem{Raney1952}
G.~N. Raney, \emph{Completely distributive complete lattices}, Proc. Amer.
  Math. Soc. \textbf{3} (1952), 677--680.

\bibitem{riehl2008factorization}
E.~Riehl, \emph{Factorization systems}, Notes available at
  \url{https://emilyriehl.github.io/files/factorization.pdf}.

\bibitem{Rossman2008}
B.~Rossman, \emph{Homomorphism preservation theorems}, J. ACM \textbf{55}
  (2008), no.~3, 15:1--15:53.

\bibitem{Tarski1955}
A.~Tarski, \emph{Contributions to the theory of models. {III}}, Nederl. Akad.
  Wetensch. Proc. Ser. A. \textbf{58} (1955), 56--64 = Indagationes Math. 17,
  56--64 (1955).

\bibitem{TZ2012}
K.~Tent and M.~Ziegler, \emph{A course in model theory}, Lecture Notes in
  Logic, vol.~40, Cambridge University Press, 2012.

\end{thebibliography}
\providecommand{\bysame}{\leavevmode\hbox to3em{\hrulefill}\thinspace}
\providecommand{\MR}{\relax\ifhmode\unskip\space\fi MR }
\providecommand{\MRhref}[2]{%
  \href{http://www.ams.org/mathscinet-getitem?mr=#1}{#2}
}
\providecommand{\href}[2]{#2}

\end{document}